\documentclass[11pt]{amsart}
\usepackage{amssymb,amsmath,amsfonts,amscd,euscript}
\usepackage[dvipsnames]{xcolor}
\usepackage{tikz}
\usepackage{hyperref}
\usepackage{mathdots}
\usepackage{enumitem}
\usepackage[dvipsnames]{xcolor}
\newcommand{\nc}{\newcommand}

\numberwithin{equation}{section}

\newtheorem{thm}{Theorem}[section]
\newtheorem{prop}[thm]{Proposition}
\newtheorem{lem}[thm]{Lemma}
\newtheorem{cor}[thm]{Corollary}

\theoremstyle{definition}
\newtheorem{dfn}[thm]{Definition}

\theoremstyle{remark}
\newtheorem{rem}[thm]{Remark}

\nc{\la}{\lambda}
\nc{\al}{\rho }
\nc{\be}{\beta }
\nc{\ve}{\varepsilon }
\nc{\om}{\omega }

\nc{\hk}{\twoheadrightarrow}
\nc{\msl}{\mathfrak{sl}}

\newcommand{\fb}{{\mathfrak b}}
\newcommand{\fr}{{\mathfrak r}}
\newcommand{\fgg}{{\mathfrak g}}

\newcommand{\Max}{\text{Max}}
\newcommand{\Omax}{\text{Omax}}
\newcommand{\omax}{\text{omax}}

\begin{document}

\title[Symmetric Dellac configurations]
{Symmetric Dellac configurations}

\author{Ange Bigeni}
\address{Ange Bigeni:\newline
Department of Mathematics,\newline
National Research University Higher School of Economics,\newline
Usacheva str. 6, 119048, Moscow, Russia}
\email{ange.bigeni@gmail.com}

\author{Evgeny Feigin}
\address{Evgeny Feigin:\newline
Department of Mathematics,\newline
National Research University Higher School of Economics,\newline
Usacheva str. 6, 119048, Moscow, Russia,\newline
{\it and }\newline
Skolkovo Institute of Science and Technology, Nobelya Ulitsa 3, Moscow 121205, Russia}
\email{evgfeig@gmail.com}

\begin{abstract}
We define symmetric Dellac configurations as the Dellac configurations that are symmetric with respect to their centers. The symmetric Dellac configurations whose lengths are even were previously introduced by Fang and Fourier under the name of symplectic Dellac configurations, to parametrize the torus fixed points of symplectic degenerate flag varieties. In general, symmetric Dellac configurations generate the Poincar\'e polynomials of (odd or even) symplectic or orthogonal versions of
the degenerate flag varieties. In this paper, we give several combinatorial interpretations of the polynomial extensions $(D_n(x))_{n \geq~0}$ of median Euler numbers, defined by Randrianarivony and Zeng, in terms of objects that we name extended Dellac configurations and which generate symmetric Dellac configurations. As a consequence, the cardinalities of the odd and even symmetric Dellac configurations are respectively given by the two adjoining sequences $(l_n)_{n \geq~0} = (1, 1, 3, 21, 267,\dots)$ and $(r_n)_{n \geq~0} = (1,2,10,98,1594,\dots)$, defined as specializations of the polynomials $(D_n(x))_{n \geq~0}$.
\end{abstract}

\maketitle

\section*{Notation}
For a pair of integers $n < m$, the set $\{n,n+1,\dots,m\}$ is denoted by $[n,m]$, and the set $[1,n]$ by $[n]$. If $T$ is a rectangular tableau, then its $j$-th column (from left to right) is denoted by $C_j^T$, and its $i$-th row (from bottom to top) is denoted by $L_i^T$. The intersection $C_j^T \cap L_i^T$ is a box that we denote by $(j:i)$. If this box contains a point $p$, we also write $p = (j:i)$. If it is known than $L_i^T$ contains one unique point, then this point is denoted by $p_i^T$.

\section{Introduction}

Let $N$ be a positive integer. Recall that a Dellac configuration~\cite{De} $D \in DC_N$ is a tableau made of $N$ columns and $2N$ rows, which contains $2N$ points (one per row, two per column) between the lines $y = x$ and $y = n+x$ (in other words, if a box $(j:i)$ of $T$ contains a point, then $j \leq i \leq n+j$). The cardinality of $DC_N$ is $h_N$ where $(h_N)_{N \geq~0} = (1,1,2,7,38,295,\dots)$~\cite{112738} is the sequence of normalized median Genocchi numbers
\cite{BD,B2,B3,B4,F1,F2,HZ1,HZ2,K}. For $D \in DC_N$, let $\text{inv}(D)$ be the number of pairs $(p_1,p_2)$ of points of $D$ such that $p_1$ is located to the left of $p_2$ and above $p_2$ (such a pair is named an \textit{inversion} of $D$). For example, we depict in Figure~\ref{fig:DC123} the elements of $(DC_N)_{N \in [3]}$ (from top to bottom), whose inversions are represented by segments connecting the concerned points.
\begin{figure}[!htbp]

\begin{center}

\begin{tikzpicture}[scale=0.4]

\begin{scope}[yshift=1cm]

\begin{scope}[xshift=11.5cm]
\draw [color=red] (-0.25,-0.25) rectangle (1.25,2.25);

\draw (0,0) grid[step=1] (1,2);
\draw (0,0) -- (1,1);
\draw (0,1) -- (1,2);
\fill (0.5,0.5) circle (0.2);
\fill (0.5,1.5) circle (0.2);
\end{scope}

\end{scope}

\begin{scope}[yshift=-4cm]

\begin{scope}[xshift=12.5cm]

\draw [color=red] (-0.25,-0.25) rectangle (2.25,4.25);

\draw (0,0) grid[step=1] (2,4);
\draw (0,0) -- (2,2);
\draw (0,2) -- (2,4);
\fill (0.5,0.5) circle (0.2);
\fill (0.5,1.5) circle (0.2);
\fill (1.5,2.5) circle (0.2);
\fill (1.5,3.5) circle (0.2);
\end{scope}

\begin{scope}[xshift=9.5cm]

\draw [color=red] (-0.25,-0.25) rectangle (2.25,4.25);

\draw (0,0) grid[step=1] (2,4);
\draw (0,0) -- (2,2);
\draw (0,2) -- (2,4);
\fill (0.5,0.5) circle (0.2);
\fill (1.5,1.5) circle (0.2);
\fill (0.5,2.5) circle (0.2);
\fill (1.5,3.5) circle (0.2);

\draw [line width=0.25mm] (0.5,2.5) -- (1.5,1.5);
\end{scope}

\end{scope}

\begin{scope}[yshift=-11cm]

\begin{scope}[xshift=21cm]

\draw [color=red] (-0.25,-0.25) rectangle (3.25,6.25);

\draw (0,0) grid[step=1] (3,6);
\draw (0,0) -- (3,3);
\draw (0,3) -- (3,6);
\fill (0.5,0.5) circle (0.2);
\fill (0.5,1.5) circle (0.2);
\fill (1.5,2.5) circle (0.2);
\fill (1.5,3.5) circle (0.2);
\fill (2.5,4.5) circle (0.2);
\fill (2.5,5.5) circle (0.2);
\end{scope}

\begin{scope}[xshift=17.5cm]
\draw (0,0) grid[step=1] (3,6);
\draw (0,0) -- (3,3);
\draw (0,3) -- (3,6);
\fill (0.5,0.5) circle (0.2);
\fill (0.5,1.5) circle (0.2);
\fill (1.5,2.5) circle (0.2);
\fill (2.5,3.5) circle (0.2);
\fill (1.5,4.5) circle (0.2);
\fill (2.5,5.5) circle (0.2);

\draw [line width=0.25mm] (1.5,4.5) -- (2.5,3.5);
\end{scope}

\begin{scope}[xshift=14cm]
\draw (0,0) grid[step=1] (3,6);
\draw (0,0) -- (3,3);
\draw (0,3) -- (3,6);
\fill (0.5,0.5) circle (0.2);
\fill (1.5,1.5) circle (0.2);
\fill (0.5,2.5) circle (0.2);
\fill (1.5,3.5) circle (0.2);
\fill (2.5,4.5) circle (0.2);
\fill (2.5,5.5) circle (0.2);

\draw [line width=0.25mm] (0.5,2.5) -- (1.5,1.5);
\end{scope}

\begin{scope}[xshift=7cm]
\draw (0,0) grid[step=1] (3,6);
\draw (0,0) -- (3,3);
\draw (0,3) -- (3,6);
\fill (0.5,0.5) circle (0.2);
\fill (0.5,1.5) circle (0.2);
\fill (2.5,2.5) circle (0.2);
\fill (1.5,3.5) circle (0.2);
\fill (1.5,4.5) circle (0.2);
\fill (2.5,5.5) circle (0.2);

\draw [line width=0.25mm] (1.5,3.5) -- (2.5,2.5);
\draw [line width=0.25mm] (1.5,4.5) -- (2.5,2.5);
\end{scope}

\begin{scope}[xshift=10.5cm]

\draw [color=red] (-0.25,-0.25) rectangle (3.25,6.25);

\draw (0,0) grid[step=1] (3,6);
\draw (0,0) -- (3,3);
\draw (0,3) -- (3,6);
\fill (0.5,0.5) circle (0.2);
\fill (1.5,1.5) circle (0.2);
\fill (0.5,2.5) circle (0.2);
\fill (2.5,3.5) circle (0.2);
\fill (1.5,4.5) circle (0.2);
\fill (2.5,5.5) circle (0.2);

\draw [line width=0.25mm] (0.5,2.5) -- (1.5,1.5);
\draw [line width=0.25mm] (1.5,4.5) -- (2.5,3.5);
\end{scope}

\begin{scope}[xshift=3.5cm]
\draw (0,0) grid[step=1] (3,6);
\draw (0,0) -- (3,3);
\draw (0,3) -- (3,6);
\fill (0.5,0.5) circle (0.2);
\fill (1.5,1.5) circle (0.2);
\fill (1.5,2.5) circle (0.2);
\fill (0.5,3.5) circle (0.2);
\fill (2.5,4.5) circle (0.2);
\fill (2.5,5.5) circle (0.2);

\draw [line width=0.25mm] (0.5,3.5) -- (1.5,1.5);
\draw [line width=0.25mm] (0.5,3.5) -- (1.5,2.5);
\end{scope}

\begin{scope}[xshift=0cm]

\draw [color=red] (-0.25,-0.25) rectangle (3.25,6.25);

\draw (0,0) grid[step=1] (3,6);
\draw (0,0) -- (3,3);
\draw (0,3) -- (3,6);
\fill (0.5,0.5) circle (0.2);
\fill (1.5,1.5) circle (0.2);
\fill (2.5,2.5) circle (0.2);
\fill (0.5,3.5) circle (0.2);
\fill (1.5,4.5) circle (0.2);
\fill (2.5,5.5) circle (0.2);

\draw [line width=0.25mm] (0.5,3.5) -- (1.5,1.5);
\draw [line width=0.25mm] (0.5,3.5) -- (2.5,2.5);
\draw [line width=0.25mm] (1.5,4.5) -- (2.5,2.5);
\end{scope}
\end{scope}
\end{tikzpicture}

\end{center}

\caption{The elements of $DC_1,DC_2,DC_3$ (the elements of $SDC_1,SDC_2,SDC_3$ are framed in red).}
\label{fig:DC123}

\end{figure}
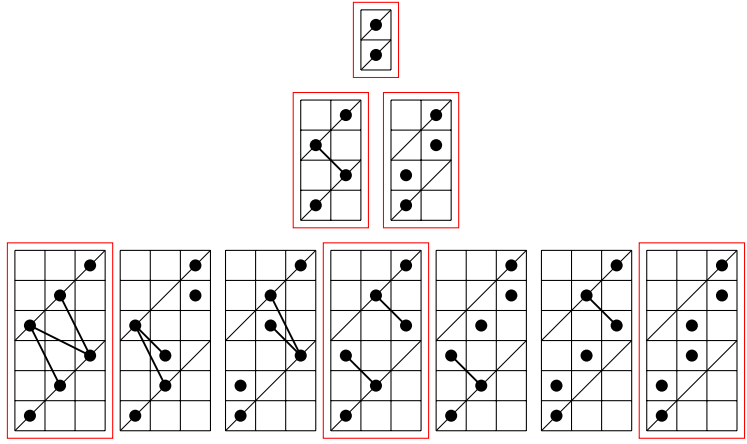
In \cite{F1} (see also \cite{F3,FFL}) it is proved that the Poincar\'e polynomial of the degenerate flag variety $\mathcal{F}_N^a$ has the following combinatorial interpretation for all $N \geq~1$~:
\begin{equation*}
P_{\mathcal{F}_N^a}(q) = \sum_{D \in DC_N} q^{\text{inv}(D)}.
\end{equation*}
For example, we can compute the Poincar\'e polynomials of $(\mathcal{F}_N^a)_{N\in[3]}$ from Figure~\ref{fig:DC123}~:
\begin{align*}
P_{\mathcal{F}_1^a}(q) &= 1,\\
P_{\mathcal{F}_2^a}(q) & = q+1,\\
P_{\mathcal{F}_3^a}(q) & = q^3+3q^2+2q+1.
\end{align*}

\subsection{Symmetric Dellac configurations}
Note that, in general, the set $DC_N$ is stable by the rotation $r_{\pi}$ of angle $\pi$ with respect to the center $(N/2,N)$ of Dellac configurations. We then define \textit{symmetric Dellac configurations} $D \in SDC_N$ as the Dellac configurations $D \in DC_N$ such that $r_{\pi}(D) = D$. In Figure~\ref{fig:DC123}, the symmetric Dellac configurations are framed in red.

Even symmetric Dellac configurations $D \in SDC_{2n}$ have been introduced in \cite{FF} under the name \textit{symplectic Dellac configurations}, to parametrize the torus fixed points of symplectic degenerate flag varieties $Sp\mathcal{F}_{2n}^a$. In \cite{BF}, we extend naturally the definition of $Sp\mathcal{F}_N^a$ for all $N \geq~1$, and we prove that the Poincar\'e polynomials of these varieties have the following combinatorial interpretation~:
\begin{equation*}
P_{Sp\mathcal{F}_{N}^a}(q) = \sum_{D \in SDC_{N}} q^{\widetilde{\text{inv}}(D)}
\end{equation*}
where $\widetilde{\text{inv}}(D)$ is the number of inversions of $D$, modulo the rotation $r_{\pi}$ (\textit{i.e.}, if two inversions of $D$ are symmetric with respect to $r_{\pi}$, that is, they are of the kind $(p_1,p_2)$ and $(r_{\pi}(p_2),r_{\pi}(p_1))$, then they represent one and only one element). For example, below are the Poincar\'e polynomials of $(Sp\mathcal{F}_N^a)_{N \in [4]}$ (the cases $N \in [3]$ are given by Figure~\ref{fig:DC123})~:
\begin{align*}
P_{Sp\mathcal{F}_1^a}(q) &= 1,\\
P_{Sp\mathcal{F}_2^a}(q) & = q+1,\\
P_{Sp\mathcal{F}_3^a}(q) & = q^2+q+1,\\
P_{Sp\mathcal{F}^a_4}(q) &= q^4 + 3q^3 + 3q^2 + 2q + 1.
\end{align*}
In \cite{BF}, we also define natural varieties $(SO\mathcal{F}_N^a)_{N\geq1}$, named \textit{orthogonal degenerate flag varieties}, whose Poincar\'e polynomials have the following combinatorial interpretation~:
\begin{equation*}
P_{SO\mathcal{F}_{N}^a}(q) = \sum_{D \in SDC_{N}} q^{\overline{\text{inv}}(D)}
\end{equation*}
where $\overline{\text{inv}}(D)$ equals $\widetilde{\text{inv}}(D)$ minus the number of inversions of $D$ of the kind $(p,r_{\pi}(p))$. For example, the Poincar\'e polynomials of $(SO\mathcal{F}_N^a)_{N\in[4]}$ (the cases $N\in[3]$ are given by Figure~\ref{fig:DC123}) are~:
\begin{align*}
P_{SO\mathcal{F}_1^a}(q) &= 1,\\
P_{SO\mathcal{F}_2^a}(q) & = 2,\\
P_{SO\mathcal{F}_3^a}(q) & = 2q+1,\\
P_{SO\mathcal{F}^a_4}(q) & = 4q^2 + 4q + 2.
\end{align*}

The primary goal of this paper is to determine the cardinalities of the sets $(SDC_N)_{N\geq~1}$. According to the parity of $N$, they appear in two adjoining sequences $(l_n)_{n \geq~0} = (1, 1, 3, 21, 267,\dots)$~\cite{11321267} and $(r_n)_{n \geq~0} = (1,2,10,98,1594,\dots)$~\cite{121098} defined in \cite{RZ} as specializations of polynomials interpolating median Euler numbers, of which we give a reminder hereafter.

\subsection{Extended median Euler numbers}
\label{sec:DnPnlnrn}

Consider the family of polynomials $(D_n(x))_{n \geq~0}$ defined by
\begin{align*}
D_0(x) &=1,\\
D_{n+1}(x) &= (x+1)(x+2)D_n(x+2) - x(x+1)D_n(x).
\end{align*}
They were introduced by Randrianarivony and Zeng~\cite{RZ} as interpolations of median Euler numbers $(L_n)_{n \geq~0} = (1, 1, 4, 46, 1024,\dots)$~\cite{114461024} and $(R_n)_{n \geq~0} = (1, 3, 24, 402,\dots)$~\cite{1324402}, through the equalities $L_n  = D_n(1/2)$ and $R_n  = D_n(-1/2)$.

It is easy to check that for all $n \geq~1$, the polynomial $D_n(x)$ is of the form $2^n(x+1)P_{n}(x+1)$ for some polynomial $P_{n}(x)$ with positive integer coefficients and degree $n-1$, whose induction formulas are
\begin{align*}
P_1(x) &=1,\\
P_{n+1}(x) &= \frac{(x+2)(x+1)}{2} P_n(x+2) - \frac{x(x-1)}{2} P_n(x).
\end{align*}
Several first values are
\begin{align*}
P_1(x) & = 1,\\
P_2(x) & = 2x+1,\\
P_3(x) & = 6x^2 + 10x + 5, \\
P_4(x) & = 24x^3 + 84x^2 + 110x + 49,\\
P_5(x) & = 120x^4 + 720x^3 + 1758x^2 + 1954x + 797.
\end{align*}
For all $n \geq~1$, let $c_{n,n-1} x^{n-1} + c_{n,n-2} x^{n-2} + \dots+ c_{n,1}x + c_{n,0}  = P_n(x)$. It is easy to check the following induction formulas : $c_{1,0} = 1$ and
\begin{align}
\label{eq:an0}
c_{n,0} &= \sum_{i = 0}^{n-2} 2^i c_{n-1,i},\\
\label{eq:ank}
c_{n,k} &= (k+1)c_{n-1,k-1} + \sum_{i = k}^{n-2} 2^{i-k} \left( \binom{i+1}{k} + 2\binom{i+1}{k-1} \right) c_{n-1,i},\\
\label{eq:anmoins1}
c_{n,n-1} &= n c_{n-1,n-2}
\end{align}
for all $n \geq~2$ and $k \in [n-2]$. In particular $c_{n,n-1} = n!$ for all $n \geq~1$.

The sequences of positive integers $(l_n)_{n \geq~0} = (1, 1, 3, 21, 267,\dots)$ and $(r_n)_{n \geq~0} = (1,2,10,98,1594\dots)$ are defined by $l_n = D_n(0)/2^n$ and $r_n = D_n(1)/2^n$ for all $n \geq~0$, which, in terms of $P_n$, is equivalent with $l_0 = r_0 = 1$ and
\begin{align*}
l_n &= P_n(1),\\
r_n &= 2P_n(2) = 2P_{n+1}(0)
\end{align*}
for all $n \geq~1$. From~\cite[Th\'eor\`eme 24]{RZ} and~\cite[Proposition 25]{RZ}, we deduce that
$$1+ \sum_{n \geq~1} xP_n(x)t^n = \dfrac{1}{1- \dfrac{xt}{1- \dfrac{(x+1)t}{1- \dfrac{2(x+2)t}{1- \dfrac{2(x+3)t}{1-\dfrac{3(x+4)t}{1-\dfrac{3(x+5)t}{\ddots}}}}}}}$$
and
\begin{equation}
\label{eq:PnintermsofSPn}
P_n(x) = \sum_{f \in SP_n} 2^{n-1-\max(f)-\text{fd}(f)} x^{\max(f)}
\end{equation}
for all $n \geq~1$, where :
\begin{itemize}
\item $SP_n$ is the set of surjective pistols \cite{Du,DR,DZ,ZZ} of size $n$ (surjective maps $f : [2n] \twoheadrightarrow \{2,4,6,\dots,2n\}$ such that 
$f(j) \geq~j$ for all $j$);
\item $\max(f)$ is the number of maximal points of $f$ (integers $j \in [2n-2]$ such that $f(j) = 2n$);
\item $\text{fd}(f)$ is the number of doubled fixed points of $f$ (integers $j \in [2n-2]$ such that there exists $j'< j$ with $f(j') = f(j) = j$).
\end{itemize}

For example $SP_2$ has three elements :
$$\begin{pmatrix}
1&2&3&4\\
2&2&4&4
\end{pmatrix},\begin{pmatrix}
1&2&3&4\\
4&2&4&4
\end{pmatrix}\text{ and }\begin{pmatrix}
1&2&3&4\\
2&4&4&4
\end{pmatrix},$$
whose numbers of maximal points are respectively 0,1 and 1, and whose number of doubled fixed points are respectively 1,0 and 0, hence $P_2(x) = 2^{1-0-1} x^0 + x^{1-1-0} x + x^{1-1-0} x = 2x+1$.

\subsection{Even symmetric Dellac configurations}
One of the goals of this paper is to give a new proof of the statement that the cardinality of $SDC_{2n}$ (for $n \geq~1)$ is the number $r_n$ from the sequence $(r_n)_{n \geq~0} = (1,2,10,98,1594,\dots)$ (another proof can be found in \cite{B1}). The first step towards this is to reduce $SDC_{2n}$ to smaller elements that we name even extended Dellac configurations, whose set is denoted by $\mathcal{T}_n^e$. The main result of this section is Theorem \ref{theo:PnintermsofTne}, which interprets $P_n(x)$ combinatorially in terms of $\mathcal{T}_n^e$, and which implies the equality $|SDC_{2n}| = 2P_n(2) = r_n$.

\subsubsection{Generation of $SDC_{2n}$}

An \textit{even extended Dellac configuration} $T \in \mathcal{T}_n^e$ is a tableau made of $n$ columns and $2n$ rows, which contains $2n$ points (one per row, two per column) above the line $y = x$ (in other words, this is the definition of $DC_n$, except that we remove the condition that the points must be under the line $y = n+x$). By considering the two points of the last column (from left to right) of a tableau $T \in \mathcal{T}_n^e$, it is easy to obtain the formula $|\mathcal{T}_n^e| = \frac{(n+1)n}{2} |\mathcal{T}_{n-1}^e| = (n+1)!n!/2^n$. For example, we represent in Figure~\ref{fig:T2e} the 3 elements of $T_2^e$. Note that the points located above the line $y = 2n-x$ (represented by a dashed line) are represented by stars ; in general, such points are named \textit{free points}, and the number of free points of $T \in \mathcal{T}_n^e$ is denoted by $\text{fr}(T)$. In Figure~\ref{fig:T2e}, the numbers of free points are 2, 1 and 2 from left to right.

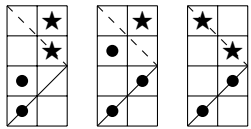
\begin{figure}[!htbp]

\begin{center}

\begin{tikzpicture}[scale=0.4]

\begin{scope}[shift={(0,0)}]

\draw (0,0) grid[step=1] (2,4);
\draw (0,0) -- (2,2);
\draw [dashed] (0,4) -- (2,2);
\fill (0.5,0.5) circle (0.2);
\fill (0.5,1.5) circle (0.2);
\draw (1.5,2.5) node[scale=0.8]{$\bigstar$};
\draw (1.5,3.5) node[scale=0.8]{$\bigstar$};

\end{scope}

\begin{scope}[shift={(3,0)}]

\draw (0,0) grid[step=1] (2,4);
\draw (0,0) -- (2,2);
\draw [dashed] (0,4) -- (2,2);
\fill (0.5,0.5) circle (0.2);
\fill (1.5,1.5) circle (0.2);
\fill (0.5,2.5) circle (0.2);
\draw (1.5,3.5) node[scale=0.8]{$\bigstar$};

\end{scope}

\begin{scope}[shift={(6,0)}]

\draw (0,0) grid[step=1] (2,4);
\draw (0,0) -- (2,2);
\draw [dashed] (0,4) -- (2,2);
\fill (0.5,0.5) circle (0.2);
\fill (1.5,1.5) circle (0.2);
\draw (1.5,2.5) node[scale=0.8]{$\bigstar$};
\draw (0.5,3.5) node[scale=0.8]{$\bigstar$};

\end{scope}

\end{tikzpicture}

\end{center}

\caption{The 3 elements of $\mathcal{T}_2^e$.}
\label{fig:T2e}

\end{figure}
There exists a one-to-one correspondence between $SDC_{2n}$ and the set of tableaux $T \in \mathcal{T}_n^e$ whose free points are labeled with the number 0 or 1 : for such a tableau $T$, the principle is :
\begin{itemize}
\item to draw an empty tableau $\widetilde{T}$ made of $2n$ columns and $4n$ rows;
\item for a box $(j:i)$ of $T$, if this box contains a point,
\begin{itemize}
\item if this point is not free, or if it is a free point whose label is $1$, then we draw a point in the box $(j:i)$ of $\widetilde{T}$;
\item otherwise this point is free and its label is $0$, and we draw a point in the box $(2n+1-j:i)$ of $\widetilde{T}$;
\end{itemize}
\item to apply the central reflexion with respect to the center $(n,2n)$ of $\widetilde{T}$ to the $2n$ points that have been drawn in $\widetilde{T}$, which indeed makes $\widetilde{T}$ an element of $SDC_{2n}$.
\end{itemize}
For example, we depict in Figure~\ref{fig:Ttilde} how a labeled tableau $T_0 \in \mathcal{T}_2^e$ is mapped to its corresponding element $\widetilde{T_0} \in SDC_4$.
\begin{figure}[!htbp]

\begin{center}

\begin{tikzpicture}[scale=0.4]

\begin{scope}[shift={(0,0)}]

\fill [color=red,opacity = 0.5] (0,3) rectangle (1,4);
\fill [color=blue,opacity = 0.5] (1,2) rectangle (2,3);

\draw (0,0) grid[step=1] (2,4);
\draw (0,0) -- (2,2);
\draw [dashed] (0,4) -- (2,2);
\fill (0.5,0.5) circle (0.2);
\fill (1.5,1.5) circle (0.2);
\draw (1.5,2.5) node[scale=1]{$0$};
\draw (0.5,3.5) node[scale=1]{$1$};

\end{scope}

\draw (3,2) node[scale=1]{$\longrightarrow$};

\begin{scope}[shift={(4,0)}]

\fill [color=red,opacity = 0.5] (0,3) rectangle (1,4);
\fill [color=blue,opacity = 0.5] (1,2) rectangle (2,3);

\fill [color=red,opacity = 0.2] (3,3) rectangle (4,4);
\fill [color=blue,opacity = 0.2] (2,2) rectangle (3,3);

\draw (0,0) grid[step=1] (2,4);
\draw (0,0) -- (4,4);
\draw (2,6) -- (4,8);
\draw [dashed] (0,4) -- (2,2);
\fill (0.5,0.5) circle (0.2);
\fill (1.5,1.5) circle (0.2);
\fill (0.5,3.5) circle (0.2);

\draw [black,opacity =1] (0,0) grid[step=1] (2,8);
\draw [black,opacity =1] (0,0) -- (2,2);
\draw [black,opacity =1] (0,4) -- (2,6);
\begin{scope}[shift={(4,8)},rotate=180]

\draw [black,opacity =1] (0,0) grid[step=1] (2,8);
\end{scope}

\fill [black,opacity = 1] (2.5,2.5) circle (0.2);

\end{scope}

\draw (9,4) node[scale=1]{$\longrightarrow$};

\begin{scope}[shift={(10,0)}]

\fill [color=red,opacity = 0.5] (0,3) rectangle (1,4);
\fill [color=blue,opacity = 0.5] (1,2) rectangle (2,3);

\fill [color=red,opacity = 0.2] (3,3) rectangle (4,4);
\fill [color=blue,opacity = 0.2] (2,2) rectangle (3,3);

\draw (0,0) grid[step=1] (2,4);
\draw (0,0) -- (2,2);
\draw [dashed] (0,4) -- (2,2);
\fill (0.5,0.5) circle (0.2);
\fill (1.5,1.5) circle (0.2);
\fill (1.5,5.5) circle (0.2);
\fill (0.5,3.5) circle (0.2);

\draw [black,opacity =1] (0,0) grid[step=1] (2,8);
\draw [black,opacity =1] (0,0) -- (2,2);
\draw [black,opacity =1] (0,4) -- (2,6);
\begin{scope}[shift={(4,8)},rotate=180]

\fill [color=red,opacity = 0.5] (0,3) rectangle (1,4);
\fill [color=blue,opacity = 0.5] (1,2) rectangle (2,3);

\fill [color=red,opacity = 0.2] (3,3) rectangle (4,4);
\fill [color=blue,opacity = 0.2] (2,2) rectangle (3,3);

\draw [black,opacity =1] (0,0) grid[step=1] (2,8);
\draw [black,opacity =1] (0,0) -- (2,2);
\draw [black,opacity =1] (0,4) -- (2,6);
\draw [black,opacity =1] [dashed] (0,4) -- (2,2);
\fill [black,opacity =1] (0.5,0.5) circle (0.2);
\fill [black,opacity =1] (1.5,1.5) circle (0.2);
\fill [black,opacity =1] (1.5,5.5) circle (0.2);
\fill [black,opacity =1] (0.5,3.5) circle (0.2);
\end{scope}

\fill [black,opacity = 1] (2.5,2.5) circle (0.2);

\end{scope}

\end{tikzpicture}

\end{center}

\caption{Construction of an even symmetric Dellac configuration.}
\label{fig:Ttilde}

\end{figure}
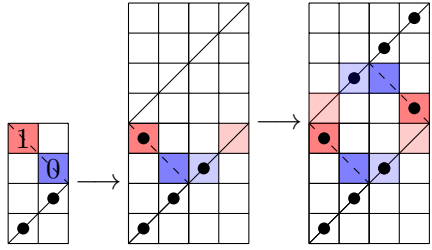

\begin{figure}[!htbp]
\begin{center}

\begin{tikzpicture}[scale=0.3]

\draw (0,0) grid[step=1] (2,4);
\draw (0,0) -- (2,2);
\draw [dashed] (0,4) -- (2,2);
\fill (0.5,0.5) circle (0.2);
\fill (0.5,1.5) circle (0.2);
\draw (1.5,2.5) node[scale=0.6]{$\bigstar$};
\draw (1.5,3.5) node[scale=0.6]{$\bigstar$};

\draw [gray,opacity = 0.7] (0,0) grid[step=1] (2,8);
\draw [gray,opacity = 0.7] (0,0) -- (2,2);
\draw [gray,opacity = 0.7] (0,4) -- (2,6);
\begin{scope}[shift={(4,8)},rotate=180]
\draw [gray,opacity = 0.7] (0,0) grid[step=1] (2,8);
\draw [gray,opacity = 0.7] (0,0) -- (2,2);
\draw [gray,opacity = 0.7] (0,4) -- (2,6);
\draw [gray,opacity = 0.7] [dashed] (0,4) -- (2,2);
\fill [gray,opacity = 0.7] (0.5,0.5) circle (0.2);
\fill [gray,opacity = 0.7] (0.5,1.5) circle (0.2);
\draw [gray,opacity = 0.7] (1.5,2.5) node[scale=0.6]{$\bigstar$};
\draw [gray,opacity = 0.7] (1.5,3.5) node[scale=0.6]{$\bigstar$};
\end{scope}

\begin{scope}[shift={(5,0)}]

\draw (0,0) grid[step=1] (2,4);
\draw (0,0) -- (2,2);
\draw [dashed] (0,4) -- (2,2);
\fill (0.5,0.5) circle (0.2);
\fill (0.5,1.5) circle (0.2);
\draw (1.5,2.5) node[scale=0.6]{$\bigstar$};
\draw (2.5,3.5) node[scale=0.6]{$\bigstar$};

\draw [gray,opacity = 0.7] (0,0) grid[step=1] (2,8);
\draw [gray,opacity = 0.7] (0,0) -- (2,2);
\draw [gray,opacity = 0.7] (0,4) -- (2,6);
\begin{scope}[shift={(4,8)},rotate=180]
\draw [gray,opacity = 0.7] (0,0) grid[step=1] (2,8);
\draw [gray,opacity = 0.7] (0,0) -- (2,2);
\draw [gray,opacity = 0.7] (0,4) -- (2,6);
\draw [gray,opacity = 0.7] [dashed] (0,4) -- (2,2);
\fill [gray,opacity = 0.7] (0.5,0.5) circle (0.2);
\fill [gray,opacity = 0.7] (0.5,1.5) circle (0.2);
\draw [gray,opacity = 0.7] (1.5,2.5) node[scale=0.6]{$\bigstar$};
\draw [gray,opacity = 0.7] (2.5,3.5) node[scale=0.6]{$\bigstar$};
\end{scope}

\end{scope}

\begin{scope}[shift={(0,-9)}]

\draw (0,0) grid[step=1] (2,4);
\draw (0,0) -- (2,2);
\draw [dashed] (0,4) -- (2,2);
\fill (0.5,0.5) circle (0.2);
\fill (0.5,1.5) circle (0.2);
\draw (2.5,2.5) node[scale=0.6]{$\bigstar$};
\draw (1.5,3.5) node[scale=0.6]{$\bigstar$};

\draw [gray,opacity = 0.7] (0,0) grid[step=1] (2,8);
\draw [gray,opacity = 0.7] (0,0) -- (2,2);
\draw [gray,opacity = 0.7] (0,4) -- (2,6);
\begin{scope}[shift={(4,8)},rotate=180]
\draw [gray,opacity = 0.7] (0,0) grid[step=1] (2,8);
\draw [gray,opacity = 0.7] (0,0) -- (2,2);
\draw [gray,opacity = 0.7] (0,4) -- (2,6);
\draw [gray,opacity = 0.7] [dashed] (0,4) -- (2,2);
\fill [gray,opacity = 0.7] (0.5,0.5) circle (0.2);
\fill [gray,opacity = 0.7] (0.5,1.5) circle (0.2);
\draw [gray,opacity = 0.7] (2.5,2.5) node[scale=0.6]{$\bigstar$};
\draw [gray,opacity = 0.7] (1.5,3.5) node[scale=0.6]{$\bigstar$};
\end{scope}

\end{scope}

\begin{scope}[shift={(5,-9)}]

\draw (0,0) grid[step=1] (2,4);
\draw (0,0) -- (2,2);
\draw [dashed] (0,4) -- (2,2);
\fill (0.5,0.5) circle (0.2);
\fill (0.5,1.5) circle (0.2);
\draw (2.5,2.5) node[scale=0.6]{$\bigstar$};
\draw (2.5,3.5) node[scale=0.6]{$\bigstar$};

\draw [gray,opacity = 0.7] (0,0) grid[step=1] (2,8);
\draw [gray,opacity = 0.7] (0,0) -- (2,2);
\draw [gray,opacity = 0.7] (0,4) -- (2,6);
\begin{scope}[shift={(4,8)},rotate=180]
\draw [gray,opacity = 0.7] (0,0) grid[step=1] (2,8);
\draw [gray,opacity = 0.7] (0,0) -- (2,2);
\draw [gray,opacity = 0.7] (0,4) -- (2,6);
\draw [gray,opacity = 0.7] [dashed] (0,4) -- (2,2);
\fill [gray,opacity = 0.7] (0.5,0.5) circle (0.2);
\fill [gray,opacity = 0.7] (0.5,1.5) circle (0.2);
\draw [gray,opacity = 0.7] (2.5,2.5) node[scale=0.6]{$\bigstar$};
\draw [gray,opacity = 0.7] (2.5,3.5) node[scale=0.6]{$\bigstar$};
\end{scope}

\end{scope}

\draw (-1,-10) rectangle (10,9);

\begin{scope}[shift={(3.5,10)}]

\draw (0,0) grid[step=1] (2,4);
\draw (0,0) -- (2,2);
\draw [dashed] (0,4) -- (2,2);
\fill (0.5,0.5) circle (0.2);
\fill (0.5,1.5) circle (0.2);
\draw (1.5,2.5) node[scale=0.6]{$\bigstar$};
\draw (1.5,3.5) node[scale=0.6]{$\bigstar$};

\end{scope}


\begin{scope}[shift={(11,0)}]

\draw (0,0) grid[step=1] (2,4);
\draw (0,0) -- (2,2);
\draw [dashed] (0,4) -- (2,2);
\fill (0.5,0.5) circle (0.2);
\fill (1.5,1.5) circle (0.2);
\fill (0.5,2.5) circle (0.2);
\draw (1.5,3.5) node[scale=0.6]{$\bigstar$};

\draw [gray,opacity = 0.7] (0,0) grid[step=1] (2,8);
\draw [gray,opacity = 0.7] (0,0) -- (2,2);
\draw [gray,opacity = 0.7] (0,4) -- (2,6);
\begin{scope}[shift={(4,8)},rotate=180]
\draw [gray,opacity = 0.7] (0,0) grid[step=1] (2,8);
\draw [gray,opacity = 0.7] (0,0) -- (2,2);
\draw [gray,opacity = 0.7] (0,4) -- (2,6);
\draw [gray,opacity = 0.7] [dashed] (0,4) -- (2,2);
\fill [gray,opacity = 0.7] (0.5,0.5) circle (0.2);
\fill [gray,opacity = 0.7] (1.5,1.5) circle (0.2);
\fill [gray,opacity = 0.7] (0.5,2.5) circle (0.2);
\draw [gray,opacity = 0.7] (1.5,3.5) node[scale=0.6]{$\bigstar$};
\end{scope}

\begin{scope}[shift={(0,-9)}]

\draw (0,0) grid[step=1] (2,4);
\draw (0,0) -- (2,2);
\draw [dashed] (0,4) -- (2,2);
\fill (0.5,0.5) circle (0.2);
\fill (1.5,1.5) circle (0.2);
\fill (0.5,2.5) circle (0.2);
\draw (2.5,3.5) node[scale=0.6]{$\bigstar$};

\draw [gray,opacity = 0.7] (0,0) grid[step=1] (2,8);
\draw [gray,opacity = 0.7] (0,0) -- (2,2);
\draw [gray,opacity = 0.7] (0,4) -- (2,6);
\begin{scope}[shift={(4,8)},rotate=180]
\draw [gray,opacity = 0.7] (0,0) grid[step=1] (2,8);
\draw [gray,opacity = 0.7] (0,0) -- (2,2);
\draw [gray,opacity = 0.7] (0,4) -- (2,6);
\draw [gray,opacity = 0.7] [dashed] (0,4) -- (2,2);
\fill [gray,opacity = 0.7] (0.5,0.5) circle (0.2);
\fill [gray,opacity = 0.7] (1.5,1.5) circle (0.2);
\fill [gray,opacity = 0.7] (0.5,2.5) circle (0.2);
\draw [gray,opacity = 0.7] (2.5,3.5) node[scale=0.6]{$\bigstar$};
\end{scope}

\end{scope}

\draw (-1,-10) rectangle (5,9);

\begin{scope}[shift={(1,10)}]

\draw (0,0) grid[step=1] (2,4);
\draw (0,0) -- (2,2);
\draw [dashed] (0,4) -- (2,2);
\fill (0.5,0.5) circle (0.2);
\fill (1.5,1.5) circle (0.2);
\fill (0.5,2.5) circle (0.2);
\draw (1.5,3.5) node[scale=0.6]{$\bigstar$};

\end{scope}

\end{scope}


\begin{scope}[shift={(17,0)}]

\draw (0,0) grid[step=1] (2,4);
\draw (0,0) -- (2,2);
\draw [dashed] (0,4) -- (2,2);
\fill (0.5,0.5) circle (0.2);
\fill (1.5,1.5) circle (0.2);
\draw (1.5,2.5) node[scale=0.6]{$\bigstar$};
\draw (0.5,3.5) node[scale=0.6]{$\bigstar$};

\draw [gray,opacity = 0.7] (0,0) grid[step=1] (2,8);
\draw [gray,opacity = 0.7] (0,0) -- (2,2);
\draw [gray,opacity = 0.7] (0,4) -- (2,6);
\begin{scope}[shift={(4,8)},rotate=180]
\draw [gray,opacity = 0.7] (0,0) grid[step=1] (2,8);
\draw [gray,opacity = 0.7] (0,0) -- (2,2);
\draw [gray,opacity = 0.7] (0,4) -- (2,6);
\draw [gray,opacity = 0.7] [dashed] (0,4) -- (2,2);
\fill [gray,opacity = 0.7] (0.5,0.5) circle (0.2);
\fill [gray,opacity = 0.7] (1.5,1.5) circle (0.2);
\draw [gray,opacity = 0.7] (1.5,2.5) node[scale=0.6]{$\bigstar$};
\draw [gray,opacity = 0.7] (0.5,3.5) node[scale=0.6]{$\bigstar$};
\end{scope}

\begin{scope}[shift={(0,-9)}]

\draw (0,0) grid[step=1] (2,4);
\draw (0,0) -- (2,2);
\draw [dashed] (0,4) -- (2,2);
\fill (0.5,0.5) circle (0.2);
\fill (1.5,1.5) circle (0.2);
\draw (2.5,2.5) node[scale=0.6]{$\bigstar$};
\draw (0.5,3.5) node[scale=0.6]{$\bigstar$};

\draw [gray,opacity = 0.7] (0,0) grid[step=1] (2,8);
\draw [gray,opacity = 0.7] (0,0) -- (2,2);
\draw [gray,opacity = 0.7] (0,4) -- (2,6);
\begin{scope}[shift={(4,8)},rotate=180]
\draw [gray,opacity = 0.7] (0,0) grid[step=1] (2,8);
\draw [gray,opacity = 0.7] (0,0) -- (2,2);
\draw [gray,opacity = 0.7] (0,4) -- (2,6);
\draw [gray,opacity = 0.7] [dashed] (0,4) -- (2,2);
\fill [gray,opacity = 0.7] (0.5,0.5) circle (0.2);
\fill [gray,opacity = 0.7] (1.5,1.5) circle (0.2);
\draw [gray,opacity = 0.7] (2.5,2.5) node[scale=0.6]{$\bigstar$};
\draw [gray,opacity = 0.7] (0.5,3.5) node[scale=0.6]{$\bigstar$};
\end{scope}

\end{scope}

\begin{scope}[shift={(5,0)}]

\draw (0,0) grid[step=1] (2,4);
\draw (0,0) -- (2,2);
\draw [dashed] (0,4) -- (2,2);
\fill (0.5,0.5) circle (0.2);
\fill (1.5,1.5) circle (0.2);
\draw (1.5,2.5) node[scale=0.6]{$\bigstar$};
\draw (3.5,3.5) node[scale=0.6]{$\bigstar$};

\draw [gray,opacity = 0.7] (0,0) grid[step=1] (2,8);
\draw [gray,opacity = 0.7] (0,0) -- (2,2);
\draw [gray,opacity = 0.7] (0,4) -- (2,6);
\begin{scope}[shift={(4,8)},rotate=180]
\draw [gray,opacity = 0.7] (0,0) grid[step=1] (2,8);
\draw [gray,opacity = 0.7] (0,0) -- (2,2);
\draw [gray,opacity = 0.7] (0,4) -- (2,6);
\draw [gray,opacity = 0.7] [dashed] (0,4) -- (2,2);
\fill [gray,opacity = 0.7] (0.5,0.5) circle (0.2);
\fill [gray,opacity = 0.7] (1.5,1.5) circle (0.2);
\draw [gray,opacity = 0.7] (1.5,2.5) node[scale=0.6]{$\bigstar$};
\draw [gray,opacity = 0.7] (3.5,3.5) node[scale=0.6]{$\bigstar$};
\end{scope}

\end{scope}

\begin{scope}[shift={(5,-9)}]

\draw (0,0) grid[step=1] (2,4);
\draw (0,0) -- (2,2);
\draw [dashed] (0,4) -- (2,2);
\fill (0.5,0.5) circle (0.2);
\fill (1.5,1.5) circle (0.2);
\draw (2.5,2.5) node[scale=0.6]{$\bigstar$};
\draw (3.5,3.5) node[scale=0.6]{$\bigstar$};

\draw [gray,opacity = 0.7] (0,0) grid[step=1] (2,8);
\draw [gray,opacity = 0.7] (0,0) -- (2,2);
\draw [gray,opacity = 0.7] (0,4) -- (2,6);
\begin{scope}[shift={(4,8)},rotate=180]
\draw [gray,opacity = 0.7] (0,0) grid[step=1] (2,8);
\draw [gray,opacity = 0.7] (0,0) -- (2,2);
\draw [gray,opacity = 0.7] (0,4) -- (2,6);
\draw [gray,opacity = 0.7] [dashed] (0,4) -- (2,2);
\fill [gray,opacity = 0.7] (0.5,0.5) circle (0.2);
\fill [gray,opacity = 0.7] (1.5,1.5) circle (0.2);
\draw [gray,opacity = 0.7] (2.5,2.5) node[scale=0.6]{$\bigstar$};
\draw [gray,opacity = 0.7] (3.5,3.5) node[scale=0.6]{$\bigstar$};
\end{scope}

\end{scope}

\draw (-1,-10) rectangle (10,9);

\begin{scope}[shift={(3.5,10)}]

\draw (0,0) grid[step=1] (2,4);
\draw (0,0) -- (2,2);
\draw [dashed] (0,4) -- (2,2);
\fill (0.5,0.5) circle (0.2);
\fill (1.5,1.5) circle (0.2);
\draw (1.5,2.5) node[scale=0.6]{$\bigstar$};
\draw (0.5,3.5) node[scale=0.6]{$\bigstar$};

\end{scope}

\end{scope}

\end{tikzpicture}

\end{center}
\caption{Construction of $SDC_4$ from $\mathcal{T}_2^e$.}
\label{fig:SDC4generatedbyT2e}

\end{figure}
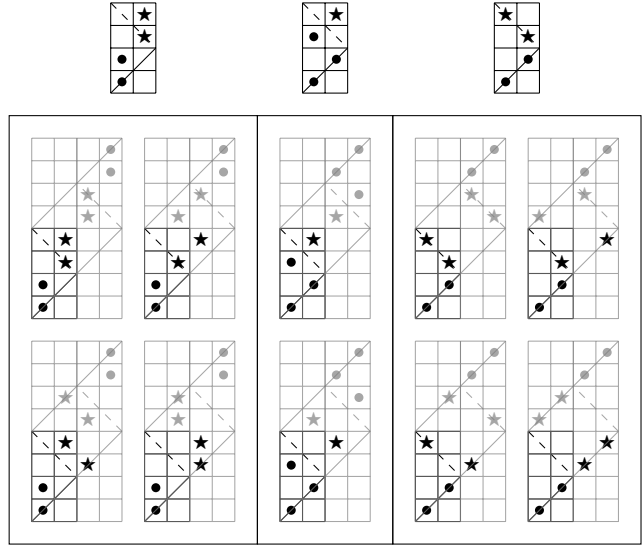

This bijection between the tableaux $T \in \mathcal{T}_n^e$ whose free points are labeled with 0 or 1 and $SDC_{2n}$ implies the following formula :
\begin{equation}
\label{eq:TnegeneratesSpDC2n}
|SDC_{2n}| = \sum_{T \in \mathcal{T}_n^e} 2^{\text{fr}(T)}.
\end{equation}
For the example $n = 2$, the $10 = 2^2+2^1+2^2$ elements of $SDC_{4}$ are generated by the 3 elements of $\mathcal{T}_2^e$ as depicted in Figure~\ref{fig:SDC4generatedbyT2e}, where the elements of $SDC_4$ in a same frame under every $T \in \mathcal{T}_2^e$ correspond with the $2^{\text{fr}(T)}$ different labellings of the free points of $T$.

\subsubsection{A combinatorial interpretation of $P_n(x)$ in terms of $\mathcal{T}_n^e$}
\label{sec:theorem1}
\begin{dfn}
\label{defi:evenpath}
Let $T \in \mathcal{T}_n^e$. For a given point $p_0 = (j:i)$ of $T$, we first define a path $(p_0,p_1,\dots)$ of points of $T$ as follows : for all $k \geq~0$, assume that the points $p_0,\dots,p_k$ are defined, and let $(j_k:i_k) = p_k$;
\begin{enumerate}[label=\alph*)]
\item if $j_k = n$, then $p_{k+1}$ is defined as $p_k$;
\item if $j_k < n$ and $p_k$ is the upper point of its column, then $p_{k+1}$ is defined as $p_{2n-j_k}^T$ (the point of the $(2n-j_k)$-th row $L_{2n-j_k}^T$ of $T$, see Notation);
\item if $j_k < n$ and $p_k$ is the lower point of its column, then $p_{k+1}$ is defined as $p_{j_k}^T$.
\end{enumerate}
We then define a finite subsequence $S_T(i) = (p_{k_0},\dots,p_{k_f})$ of $(p_0,p_1,\dots)$ by $k_0 = 0$ and, if we assume that $k_0,\dots,k_q$ are defined for some $q \geq~0$,
\begin{itemize}
\item if $j_{k_q} < \max \{j_k : k \geq~0\}$, then $k_{q+1} = \min \{k > k_q : j_k > j_{k_q}\}$;
\item otherwise $f = q$.
\end{itemize}
We finally define $B(T) = S_T(n)$ and $R(T) = S_T(2n)$. Let $\text{b}(T)+1$ and $\text{r}(T)+1$ be their respective number of elements.
\end{dfn}

For example, we represent in Figure~\ref{fig:examplepaths} a tableau $T_1 \in \mathcal{T}_7^e$ with $B(T_1) = ((2:7),(3:12),(7:11))$ and $R(T_1) = ((4:14),(7:10))$. In general, when depicting a tableau $T \in \mathcal{T}_n^e$, the $\text{b}(T)$ first elements of $B(T)$ and the $\text{r}(T)$ first elements of $R(T)$ are painted in blue and red respectively. The meaning of the green color in Figure~\ref{fig:examplepaths} will be explained afterwards (see Definition~\ref{defi:green}).

\begin{figure}[!htbp]

\begin{center}

\begin{tikzpicture}[scale=0.5]
\draw (-0.5,0.5) node[scale=1]{$1$};
\draw (-0.5,1.5) node[scale=1]{$2$};
\draw (-0.5,2.5) node[scale=1]{$3$};
\draw (-0.5,3.5) node[scale=1]{$4$};
\draw (-0.5,4.5) node[scale=1]{$5$};
\draw (-0.5,5.5) node[scale=1]{$6$};
\draw (-0.5,6.5) node[scale=1]{$7$};
\draw (-0.5,7.5) node[scale=1]{$8$};
\draw (-0.5,8.5) node[scale=1]{$9$};
\draw (-0.5,9.5) node[scale=1]{$10$};
\draw (-0.5,10.5) node[scale=1]{$11$};
\draw (-0.5,11.5) node[scale=1]{$12$};
\draw (-0.5,12.5) node[scale=1]{$13$};
\draw (-0.5,13.5) node[scale=1]{$14$};

\draw (0.5,-0.5) node[scale=1]{$1$};
\draw (1.5,-0.5) node[scale=1]{$2$};
\draw (2.5,-0.5) node[scale=1]{$3$};
\draw (3.5,-0.5) node[scale=1]{$4$};
\draw (4.5,-0.5) node[scale=1]{$5$};
\draw (5.5,-0.5) node[scale=1]{$6$};
\draw (6.5,-0.5) node[scale=1]{$7$};

\draw (0,0) grid[step=1] (7,14);
\draw (0,0) -- (7,7);
\draw [dashed] (0,14) -- (7,7);
\fill (0.5,0.5) circle (0.2);
\fill (0.5,1.5) circle (0.2);
\fill (2.5,2.5) circle (0.2);
\fill (3.5,3.5) circle (0.2);
\fill (4.5,4.5) circle (0.2);
\fill (4.5,5.5) circle (0.2);
\fill (1.5,6.5) circle (0.2);
\fill (5.5,7.5) circle (0.2);
\draw (5.5,8.5) node[scale=0.8]{$\bigstar$};
\draw (6.5,9.5) node[scale=0.8]{$\bigstar$};
\draw (6.5,10.5) node[scale=0.8]{$\bigstar$};
\draw (2.5,11.5) node[scale=0.8]{$\bigstar$};
\draw (1.5,12.5) node[scale=0.8]{$\bigstar$};
\draw (3.5,13.5) node[scale=0.8]{$\bigstar$};

\draw[color=red] (3.5,13.5) node[scale=0.8]{$\bigstar$};
\draw[color=red] [very thick] (3.5,13.5) [densely dashed] -- (3.5,9.5);
\draw[color=red] [very thick] (3.5,9.5) [densely dashed] -- (6.5,9.5);

\fill[color=blue] (1.5,6.5) circle (0.2);
\draw[color=blue] [very thick] (1.5,6.5) [densely dashed] -- (1.5,1.5);
\draw[color=blue] [very thick] (1.5,1.5) [densely dashed] -- (0.5,1.5);
\draw[color=blue] [very thick] (0.5,1.5) [densely dashed] -- (0.5,12.5);
\draw[color=blue] [very thick] (0.5,12.5) [densely dashed] -- (1.5,12.5);
\draw[color=blue] [very thick] (1.5,12.5) [densely dashed] -- (1.5,11.5);
\draw[color=blue] [very thick] (1.5,11.5) [densely dashed] -- (2.5,11.5);
\draw[color=blue] (2.5,11.5) node[scale=0.8]{$\bigstar$};
\draw[color=blue] [very thick] (2.5,11.5) [densely dashed] -- (2.5,10.5);
\draw[color=blue] [very thick] (2.5,10.5) [densely dashed] -- (6.5,10.5);

\draw[color=ForestGreen] [very thick] (1.3,1.3) [densely dashed] -- (0.3,1.3);
\draw[color=ForestGreen] [very thick] (0.3,1.3) [densely dashed] -- (0.3,12.7);
\draw[color=ForestGreen] [very thick] (0.3,12.7) [densely dashed] -- (1.3,12.7);
\draw[color=ForestGreen] (1.5,12.5) node[scale=0.8]{$\bigstar$};

\end{tikzpicture}

\end{center}

\caption{Extended Dellac configuration $T_1 \in \mathcal{T}_7^e$ such that $(\text{b}(T_1),\text{r}(T_1),\text{g}(T_1)) = (\textcolor{blue}{2},\textcolor{red}{1},\textcolor{ForestGreen}{1})$.}
\label{fig:examplepaths}

\end{figure}
\begin{dfn}
\label{defi:TPath}
Let $j \in [n]$ and consider a tableau $T$ made of $n$ columns (each of which contains at most two points) and $2n$ rows (each of which contains at most one point), such that each of the $j-1$ first rows (from bottom to top) of $T$ contains exactly one point above the line $y = x$, and each of the $j-1$ first columns (from left to right) contains exactly two points (for example, any tableau $T \in \mathcal{T}_n^e$ satisfies this condition). Let $i \in [j,2n]$ be such that the $j-1$ first boxes of the $L_i^T$ are empty (for example, any $i \in [j,2n]$ such that the box $(j:i)$ of $T$ contains a point). We define a sequence $\mathcal{I}_T(j:i) = (i_0,i_1,\dots)$ of elements of $[2n]$ as follows~: $i_0 = i$, afterwards if $i_0,\dots,i_k$ are defined for some $k \geq~0$,
\begin{enumerate}
\item if $i_k \in [j,2n-j] \sqcup \{2n\}$, then $i_{k+1}$ is defined as $i_k$;
\item if $i_k \in [2n-j+1,2n-1]$, let $j_k = 2n - i_k \in [j-1]$, then $i_{k+1}$ is defined by $(j_k:i_{k+1})$ being the upper point of $C_{j_k}^T$;
\item otherwise $i_k \in [j-1]$, and $i_{k+1}$ is defined by $(i_k:i_{k+1})$ being the lower point of the $C_{i_k}^T$.
\end{enumerate}
\end{dfn}

For example, in Figure~\ref{fig:examplepaths}, one can read the sequences corresponding to the free points $p = (j:i)$ that appear in the red, blue or green path, by following the path in the reversed direction $(\dots,p_{k},p_{k-1},\dots)$. We have $\mathcal{I}_{T_1}(7:10) = (10,14,14,\dots),\mathcal{I}_{T_1}(7:11)=(11,12,13,2,7,7,\dots)$, etc. The green path, as a subpath of the blue path, also gives $\mathcal{I}_{T_1}(2:13) = (13,2,2,\dots)$.

\begin{prop}
\label{prop:TPathstationnaire}
In the context of Definition~\ref{defi:TPath}, the sequence $\mathcal{I}_T(j:i) = (i_0,i_1,\dots)$ becomes stationary, \textit{i.e.}, there exists an integer $\text{root}_T(j:i) \in [j,2n-j] \sqcup \{2n\}$ (named the root of the box $(j:i)$ in $T$) and $k_0 \geq~0$ such that $i_k = \text{root}_T(j:i)$ for all $k \geq~k_0$. Moreover, the map $i \mapsto \text{root}_T(j:i)$ is a bijection from the set of integers $i \in [j,2n]$ such that the first $j-1$ boxes of $L_i^T$ are empty, to $[j,2n-j] \sqcup \{2n\}$.
\end{prop}

\begin{proof}
Since $[2n]$ is a finite set, there exist $0 \leq k_1 < k_2$ such that $i_{k_1} = i_{k_2}$. Suppose that $i_k \not\in [j,2n-j] \sqcup \{2n\}$ for all $k \geq~0$, \textit{i.e.}, the rule (i) of Definition~\ref{defi:TPath} is never applied. Then, the sequence $(i_k)_{k \geq~0}$ is reversible : for all $k > 0$, let $j_{k-1} \in [n]$ such that $p_{i_k}^T = (j_{k-1}:i_k)$ ; if $(j_{k-1}:i_k)$ is the upper point of its column, then $i_{k-1} = 2n-j_{k-1}$, otherwise $i_{k-1} = j_{k-1}$. This reversibility and the equality $i_{k_2} = i_{k_1}$ imply $i_{k_2-k_1} = i_{k_1-k_1} = i_0 = i$. Since $k_2 - k_1 > 0$ and, for all $k > 0$, the point $p_{i_k}^T$ is of the kind $(j_{k-1}:i_k)$ for some $j_{k-1} \in [j-1]$, then the point $p = (j:i)$ equals $(j_{k_2-k_1-1}:i_{k_2-k_1})$ for some $j_{k_2-k_1-1} < j$, which is absurd, so $i_k$ belongs to $[j,2n-j] \sqcup \{2n\}$ for $k$ big enough.

Let $k_{\min}$ be the smallest integer $k \geq~0$ such that $i_k = \text{root}_T(j:i)$. As stated in the previous paragraph, the sequence $(i=i_0,i_1,\hdots,i_{k_{\min}})$ is reversible because it never involves the rule (i) of Definition~\ref{defi:TPath}, so the map $i \mapsto \text{root}_T(j:i)$ is injective. Finally, the number of integers $i \in [j,2n]$ such that the first $j-1$ boxes of $L_i^T$ are empty is exactly $2(n-j+1) = |[j,2n-j] \sqcup \{2n\}|$ (hence $i \mapsto \text{root}_T(j:i)$ is bijective)~: since the $j-1$ first rows of $T$ contain exactly $j-1$ points, and the $j-1$ first columns of $T$ contain exactly $2(j-1)$ points, then, among the $2n-j+1$ upper rows of $T$ (the rows $C_j^T,C_{j+1}^T,\dots,C_{2n}^T$), exactly $j-1$ of them contain a point in one of their $j-1$ first boxes, so $2n-j+1 - (j-1) = 2(n-j+1)$ of them have their $j-1$ first boxes empty.
\end{proof}

\begin{dfn}
\label{defi:green}
We define $\mathcal{G}(T)$ as the set of the free points $p = (j:i)$ of $T$ such that :
\begin{itemize}
\item the point distinct from $p$ in $C_j^T$ is an element of $B(T)$;
\item the integer $\text{root}_T(j:i)$ equals $j$.
\end{itemize}
The cardinality of $\mathcal{G}(T)$ is denoted by $\text{g}(T)$ and its elements are painted in green in the graphical representation of $T$.
\end{dfn}
For example, in Figure~\ref{fig:examplepaths}, we have $\mathcal{G}(T_1) = \{(2:13)\}$.

\begin{dfn}
Let $n \geq~1$ and $T \in \mathcal{T}_n^e$. We denote by $\text{Omax}(T)$ (whose cardinality is $\text{omax}(T)$) the set of the elements of $B(T),R(T)$ and $\mathcal{G}(T)$. Note that $\text{omax}(T) \geq~2$. The set $\text{Max}(T)$ (whose cardinality is $\max(T)$) is defined as $\text{Omax}(T)$ from which we removed the last element of $B(T)$ and the last element of $R(T)$, hence $$\max(T) = b(T)+r(T)+g(T) = \omax(T)-2.$$
\end{dfn}

\begin{rem}
\label{rem:evenmaximalfreepoints}
For all $T \in \mathcal{T}_n^e$, the $\text{b}(T)$ last points of $B(T)$, the $\text{r}(T)+1$ points of $R(T)$, and the $\text{g}(T)$ points of $\mathcal{G}(T)$, are pairwise distinct free points of $T$, so
$$\text{fr}(T) \geq~\max(T)+1.$$
\end{rem}

The first main result of this paper is the following.

\begin{thm}
\label{theo:PnintermsofTne}
For all $n \geq~1$,
\begin{equation*}
\label{eq:PnintermsofTne}
P_n(x) = \sum_{T \in \mathcal{T}_n^e} 2^{\text{fr}(T)-1-\max(T)} x^{\max(T)}.
\end{equation*}
\end{thm}
For example, the three elements of $\mathcal{T}_2^e$ below
\begin{center}

\begin{tikzpicture}[scale=0.4]

\draw (0,0) grid[step=1] (2,4);
\draw (0,0) -- (2,2);
\draw [dashed] (0,4) -- (2,2);
\fill (0.5,0.5) circle (0.2);
\fill [color=blue] (0.5,1.5) circle (0.2);
\draw (1.5,2.5) node[scale=0.8]{$\bigstar$};
\draw (1.5,3.5) node[scale=0.8]{$\bigstar$};

\draw[color=black] (1.5,3.5) node[scale=0.8]{$\bigstar$};
\draw[color=blue] [very thick] (0.5,1.5) [densely dashed] -- (0.5,2.5);
\draw[color=blue] [very thick] (0.5,2.5) [densely dashed] -- (1.5,2.5);
\draw[color=black] (1.5,2.5) node[scale=0.8]{$\bigstar$};

\begin{scope}[shift={(3,0)}]

\draw (0,0) grid[step=1] (2,4);
\draw (0,0) -- (2,2);
\draw [dashed] (0,4) -- (2,2);
\fill (0.5,0.5) circle (0.2);
\fill (0.5,2.5) circle (0.2);
\fill (1.5,1.5) circle (0.2);
\draw (1.5,3.5) node[scale=0.8]{$\bigstar$};

\draw[color=black] (1.5,3.5) node[scale=0.8]{$\bigstar$};

\end{scope}

\begin{scope}[shift={(6,0)}]

\draw (0,0) grid[step=1] (2,4);
\draw (0,0) -- (2,2);
\draw [dashed] (0,4) -- (2,2);
\fill (0.5,0.5) circle (0.2);
\fill (1.5,1.5) circle (0.2);
\draw (0.5,3.5) node[scale=0.8]{$\bigstar$};
\draw (1.5,2.5) node[scale=0.8]{$\bigstar$};

\draw[color=red] (0.5,3.5) node[scale=0.8]{$\bigstar$};
\draw[color=red] [very thick] (0.5,3.5) [densely dashed] -- (0.5,2.5);
\draw[color=red] [very thick] (0.5,2.5) [densely dashed] -- (1.5,2.5);
\draw[color=black] (1.5,2.5) node[scale=0.8]{$\bigstar$};

\end{scope}
\end{tikzpicture}

\end{center}
give $$P_2(x) = 2^{2-1-\textcolor{blue}{1}-\textcolor{red}{0}-\textcolor{ForestGreen}{0}} x^{\textcolor{blue}{1}+\textcolor{red}{0}+\textcolor{ForestGreen}{0}} + 2^{1-1-\textcolor{blue}{0}-\textcolor{red}{0}-\textcolor{ForestGreen}{0}} x^{\textcolor{blue}{0}+\textcolor{red}{0}+\textcolor{ForestGreen}{0}}+2^{2-1-\textcolor{blue}{0}-\textcolor{red}{1}-\textcolor{ForestGreen}{0}} x^{\textcolor{blue}{0}+\textcolor{red}{1}+\textcolor{ForestGreen}{0}}
 = 2x+1.$$

The cardinality of $SDC_{2n}$ being $r_n = 2P_n(2)$ is then obtained by setting $x = 2$ in Theorem \ref{theo:PnintermsofTne} in view of formula~\eqref{eq:TnegeneratesSpDC2n}. Section \ref{sec:PnintermsofTne} is dedicated to the proof of Theorem \ref{theo:PnintermsofTne} (by induction, using the inductions formulas \eqref{eq:an0},\eqref{eq:ank} and \eqref{eq:anmoins1}).

\begin{rem}
An independent proof of $|SDC_{2n}| = r_n$ is given in \cite{B1}. It consists in defining a surjection $\varphi : \mathcal{T}_n^e \twoheadrightarrow SP_n$ such that
\begin{equation}
\label{eq:orbitofapistol}
2^{n-\text{fd}(f)} = \sum_{T \in \varphi^{-1}(f)} 2^{\text{fr}(T)}
\end{equation}
for all $f \in SP_n$, which proves indeed $|SDC_{2n}| = 2P_n(2) = r_n$ in view of formula~\eqref{eq:PnintermsofSPn} and formula~\eqref{eq:TnegeneratesSpDC2n}. Afterwards, one can check that for all $f \in SP_n$ and $T \in \varphi^{-1}(f)$,
\begin{enumerate}
\item $\text{g}(T)$ is the number of integers $j \in \{2,4,\dots,2n-2\}$ such that \linebreak$f(j-1) = f(j) = 2n$ and $\min\{j'<j-1 : f(j') = j\}$ is even;
\item $\text{r}(T)$ is the number of all other integers $j \in \{2,4,\dots,2n-2\}$ such that $f(j) = 2n$;
\item $\text{b}(T)$ is the number of integers $j \in \{1,3,\dots,2n-3\}$ such that $f(j) = 2n$.
\end{enumerate}
It follows that $\max(f) = \max(T)$, which proves Theorem \ref{theo:PnintermsofTne} by formula~\eqref{eq:PnintermsofSPn} and formula~\eqref{eq:orbitofapistol}
\end{rem}

\subsection{Odd symmetric Dellac configurations}
A second goal is to prove that the cardinality of $SDC_{2n-1} (n \geq~1)$ is the number $l_n$ from the sequence $(l_n)_{n \geq~0} = (1,1,3,21,267,\dots)$ introduced in Section \ref{sec:DnPnlnrn}. It is obviously true for $n = 1$ (see Figure~\ref{fig:DC123}). Afterwards, as we did for the even case, we reduce $SDC_{2n-1} (n \geq~2)$ to smaller elements called odd extended Dellac configurations, whose set is denoted by $\mathcal{T}_{n-1}^o$. In Theorem~\ref{theo:PnintermsofTno}, we give a combinatorial interpretation of $P_n(x)$ in terms of $\mathcal{T}_{n-1}^o$, which implies the equality $|SDC_{2n-1}| = P_n(1) = l_n$.

\subsubsection{Generation of $SDC_{2n+1} (n \geq~1)$}
We define \textit{odd extended Dellac configurations} $T \in \mathcal{T}_n^o$ ($n \geq~1$) as the rectangular tableaux $T$ of size $n \times (2n+1)$ with the following conditions~:
\begin{itemize}
\item every column of $T$ contains exactly two points;
\item every row of $T$ contains at most one point;
\item each of the first $n$ rows of $T$ contains exactly one point.
\end{itemize}
In other words, a tableau $T \in \mathcal{T}_n^o$ is like an element of $\mathcal{T}_n^e$, except that it has $2n+1$ rows instead of $2n$, and exactly one of its $n+1$ upper rows is empty (\textit{i.e.}, has no points). There is a natural surjection $\phi : \mathcal{T}_n^o \twoheadrightarrow \mathcal{T}_n^e$ which consists in deleting the empty row of the elements of $\mathcal{T}_n^o$, and it is clear that $|\phi^{-1}(T)| = n+1$ for all $T \in \mathcal{T}_n^e$, hence $|\mathcal{T}_n^o| = (n+1) |\mathcal{T}_n^e| = ((n+1)!)^2/2^n$. For example, the $9$ elements of $\mathcal{T}_2^o$ are surjected into the $3$ elements of $\mathcal{T}_2^e$ as follows~:
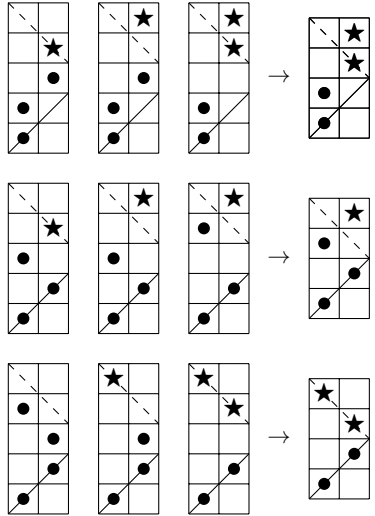
\begin{center}

\begin{tikzpicture}[scale=0.4]

\draw (0,0) grid[step=1] (2,5);
\draw (0,0) -- (2,2);
\draw [dashed] (0,5) -- (2,3);
\fill (0.5,0.5) circle (0.2);
\fill (0.5,1.5) circle (0.2);
\fill (1.5,2.5) circle (0.2);
\draw (1.5,3.5) node[scale=0.8]{$\bigstar$};

\begin{scope}[shift={(3,0)}]

\draw (0,0) grid[step=1] (2,5);
\draw (0,0) -- (2,2);
\draw [dashed] (0,5) -- (2,3);
\fill (0.5,0.5) circle (0.2);
\fill (0.5,1.5) circle (0.2);
\fill (1.5,2.5) circle (0.2);
\draw (1.5,4.5) node[scale=0.8]{$\bigstar$};

\end{scope}

\begin{scope}[shift={(6,0)}]

\draw (0,0) grid[step=1] (2,5);
\draw (0,0) -- (2,2);
\draw [dashed] (0,5) -- (2,3);
\fill (0.5,0.5) circle (0.2);
\fill (0.5,1.5) circle (0.2);
\draw (1.5,3.5) node[scale=0.8]{$\bigstar$};
\draw (1.5,4.5) node[scale=0.8]{$\bigstar$};

\end{scope}

\draw (9,2.5) node[scale=0.8]{$\rightarrow$};

\begin{scope}[shift={(10,0.5)}]

\draw (0,0) grid[step=1] (2,4);
\draw (0,0) -- (2,2);
\draw [dashed] (0,4) -- (2,2);
\fill (0.5,0.5) circle (0.2);
\fill (0.5,1.5) circle (0.2);
\draw (1.5,2.5) node[scale=0.8]{$\bigstar$};
\draw (1.5,3.5) node[scale=0.8]{$\bigstar$};

\draw (2.3,2) node[scale=1]{$,$};

\end{scope}

\begin{scope}[shift={(0,-6)}]

\draw (0,0) grid[step=1] (2,5);
\draw (0,0) -- (2,2);
\draw [dashed] (0,5) -- (2,3);
\fill (0.5,0.5) circle (0.2);
\fill (0.5,2.5) circle (0.2);
\fill (1.5,1.5) circle (0.2);
\draw (1.5,3.5) node[scale=0.8]{$\bigstar$};

\begin{scope}[shift={(3,0)}]

\draw (0,0) grid[step=1] (2,5);
\draw (0,0) -- (2,2);
\draw [dashed] (0,5) -- (2,3);
\fill (0.5,0.5) circle (0.2);
\fill (0.5,2.5) circle (0.2);
\fill (1.5,1.5) circle (0.2);
\draw (1.5,4.5) node[scale=0.8]{$\bigstar$};

\end{scope}

\begin{scope}[shift={(6,0)}]

\draw (0,0) grid[step=1] (2,5);
\draw (0,0) -- (2,2);
\draw [dashed] (0,5) -- (2,3);
\fill (0.5,0.5) circle (0.2);
\fill (1.5,1.5) circle (0.2);
\fill (0.5,3.5) circle (0.2);
\draw (1.5,4.5) node[scale=0.8]{$\bigstar$};

\end{scope}

\draw (9,2.5) node[scale=0.8]{$\rightarrow$};

\begin{scope}[shift={(10,0.5)}]

\draw (0,0) grid[step=1] (2,4);
\draw (0,0) -- (2,2);
\draw [dashed] (0,4) -- (2,2);
\fill (0.5,0.5) circle (0.2);
\fill (0.5,2.5) circle (0.2);
\fill (1.5,1.5) circle (0.2);
\draw (1.5,3.5) node[scale=0.8]{$\bigstar$};

\draw (2.3,2) node[scale=1]{$,$};

\end{scope}

\end{scope}

\begin{scope}[shift={(10,0.5)}]

\draw (0,0) grid[step=1] (2,4);
\draw (0,0) -- (2,2);
\draw [dashed] (0,4) -- (2,2);
\fill (0.5,0.5) circle (0.2);
\fill (0.5,1.5) circle (0.2);
\draw (1.5,2.5) node[scale=0.8]{$\bigstar$};
\draw (1.5,3.5) node[scale=0.8]{$\bigstar$};

\end{scope}

\begin{scope}[shift={(0,-12)}]

\draw (0,0) grid[step=1] (2,5);
\draw (0,0) -- (2,2);
\draw [dashed] (0,5) -- (2,3);
\fill (0.5,0.5) circle (0.2);
\fill (1.5,1.5) circle (0.2);
\fill (0.5,3.5) circle (0.2);
\fill (1.5,2.5) circle (0.2);

\begin{scope}[shift={(3,0)}]

\draw (0,0) grid[step=1] (2,5);
\draw (0,0) -- (2,2);
\draw [dashed] (0,5) -- (2,3);
\fill (0.5,0.5) circle (0.2);
\fill (1.5,1.5) circle (0.2);
\draw (0.5,4.5) node[scale=0.8]{$\bigstar$};
\fill (1.5,2.5) circle (0.2);

\end{scope}

\begin{scope}[shift={(6,0)}]

\draw (0,0) grid[step=1] (2,5);
\draw (0,0) -- (2,2);
\draw [dashed] (0,5) -- (2,3);
\fill (0.5,0.5) circle (0.2);
\fill (1.5,1.5) circle (0.2);
\draw (0.5,4.5) node[scale=0.8]{$\bigstar$};
\draw (1.5,3.5) node[scale=0.8]{$\bigstar$};

\end{scope}

\draw (9,2.5) node[scale=0.8]{$\rightarrow$};

\begin{scope}[shift={(10,0.5)}]

\draw (0,0) grid[step=1] (2,4);
\draw (0,0) -- (2,2);
\draw [dashed] (0,4) -- (2,2);
\fill (0.5,0.5) circle (0.2);
\fill (1.5,1.5) circle (0.2);
\draw (0.5,3.5) node[scale=0.8]{$\bigstar$};
\draw (1.5,2.5) node[scale=0.8]{$\bigstar$};

\draw (2.3,2) node[scale=1]{$.$};

\end{scope}

\end{scope}

\end{tikzpicture}

\end{center}
Note that when depicting the tableaux $T \in \mathcal{T}_n^o$, we represent with stars the points of $T$ located on or above the line $y = 2n+1-x$ (represented by a dashed line), which we also name \textit{free points} of $T$ as for the elements of $\mathcal{T}_n^e$, and whose number is also denoted by $\text{fr}(T)$.

Similarly as for $SDC_{2n}$, the odd symmetric Dellac configurations $D \in SDC_{2n+1}$ are in one-to-one correspondence with the elements of $\mathcal{T}_n^o$ whose free points are labeled with the integers 0 or 1 : for such a tableau $T$, we construct the corresponding element of $SDC_{2n+1}$ as follows~:
\begin{itemize}
\item we draw an empty tableau $\overline{T}$ made of $2n+1$ columns and $4n+2$ rows;
\item for a box $(j:i)$ of $T$, if this box contains a point,
\begin{itemize}
\item if this point is not free, or if it is a free point whose label is $1$, then we draw a point in the box $(j:i)$ of $\overline{T}$;
\item otherwise this point is free and its label is $0$, and we draw a point in the box $(2n+2
-j:i)$ of $\overline{T}$;
\end{itemize}
\item we draw a point in the box $(n+1:i_0)$ where $L_{i_0}^T$ is its unique empty row;
\item we apply the central reflexion with respect to the center \linebreak$(n+1/2,2n+1)$ of $\overline{T}$ to the $2n+1$ points that have been drawn in $\overline{T}$, which indeed makes $\overline{T}$ an element of $SDC_{2n+1}$.
\end{itemize}
For example, we depict in Figure~\ref{fig:Ttilde2} how a labeled tableau $T_2 \in \mathcal{T}_2^o$ is mapped to its corresponding element $\overline{T_2} \in SDC_5$.
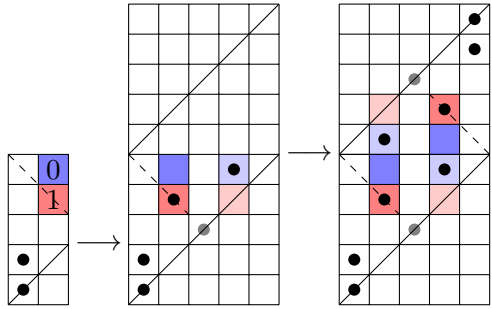
\begin{figure}[!htbp]

\begin{center}

\begin{tikzpicture}[scale=0.4]

\begin{scope}[shift={(0,0)}]

\fill [color=red,opacity = 0.5] (1,3) rectangle (2,4);
\fill [color=blue,opacity = 0.5] (1,4) rectangle (2,5);

\draw (0,0) grid[step=1] (2,5);
\draw (0,0) -- (2,2);
\draw [dashed] (0,5) -- (2,3);
\fill (0.5,0.5) circle (0.2);
\fill (0.5,1.5) circle (0.2);
\draw (1.5,3.5) node[scale=1]{$1$};
\draw (1.5,4.5) node[scale=1]{$0$};

\end{scope}

\draw (3,2) node[scale=1]{$\longrightarrow$};

\begin{scope}[shift={(4,0)}]

\fill [color=red,opacity = 0.5] (1,3) rectangle (2,4);
\fill [color=blue,opacity = 0.5] (1,4) rectangle (2,5);

\fill [color=red,opacity = 0.2] (3,3) rectangle (4,4);
\fill [color=blue,opacity = 0.2] (3,4) rectangle (4,5);

\draw (0,0) grid[step=1] (5,10);
\draw (0,0) -- (5,5);
\draw (0,5) -- (5,10);
\draw [dashed] (0,5) -- (2,3);
\fill (0.5,0.5) circle (0.2);
\fill (0.5,1.5) circle (0.2);
\fill (1.5,3.5) circle (0.2);
\fill (3.5,4.5) circle (0.2);
\fill [opacity = 0.5] (2.5,2.5) circle (0.2);

\end{scope}

\draw (10,5) node[scale=1]{$\longrightarrow$};

\begin{scope}[shift={(11,0)}]

\fill [color=red,opacity = 0.5] (1,3) rectangle (2,4);
\fill [color=blue,opacity = 0.5] (1,4) rectangle (2,5);

\fill [color=red,opacity = 0.2] (3,3) rectangle (4,4);
\fill [color=blue,opacity = 0.2] (3,4) rectangle (4,5);

\draw (0,0) grid[step=1] (5,5);
\draw (0,0) -- (5,5);
\draw [dashed] (0,5) -- (2,3);
\fill (0.5,0.5) circle (0.2);
\fill (0.5,1.5) circle (0.2);
\fill (1.5,3.5) circle (0.2);
\fill (3.5,4.5) circle (0.2);
\fill [opacity = 0.5] (2.5,2.5) circle (0.2);

\begin{scope}[shift={(5,10)}]
\begin{scope}[rotate=180]

\fill [color=red,opacity = 0.5] (1,3) rectangle (2,4);
\fill [color=blue,opacity = 0.5] (1,4) rectangle (2,5);

\fill [color=red,opacity = 0.2] (3,3) rectangle (4,4);
\fill [color=blue,opacity = 0.2] (3,4) rectangle (4,5);

\draw (0,0) grid[step=1] (5,5);
\draw (0,0) -- (5,5);
\draw [dashed] (0,5) -- (2,3);
\fill (0.5,0.5) circle (0.2);
\fill (0.5,1.5) circle (0.2);
\fill (1.5,3.5) circle (0.2);
\fill (3.5,4.5) circle (0.2);
\fill [opacity = 0.5] (2.5,2.5) circle (0.2);

\end{scope}
\end{scope}

\end{scope}

\end{tikzpicture}

\end{center}

\caption{Construction of an odd symmetric Dellac configuration.}
\label{fig:Ttilde2}

\end{figure}
This bijection between the tableaux $T \in \mathcal{T}_n^o$ whose free points are labeled with 0 or 1 and $SDC_{2n+1}$ implies the following formula for all $n \geq~1$ :
\begin{equation}
\label{eq:TnogeneratesSpDC2nmoins1}
|SDC_{2n+1}| = \sum_{T \in \mathcal{T}_n^o} 2^{\text{fr}(T)}.
\end{equation}
For the example $n = 1$, the $3 = 2^0+2^1$ elements of $SDC_{3}$ are generated by the 2 elements of $\mathcal{T}_2^o$ as depicted in Figure~\ref{fig:SDC3generatedbyT2o}, where the elements of $SDC_3$ in a same frame under every $T \in \mathcal{T}_1^o$ correspond with the $2^{\text{fr}(T)}$ different labellings of the free points of $T$.

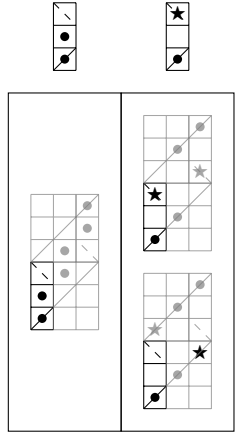
\begin{figure}[!htbp]

\begin{tikzpicture}[scale=0.3]

\begin{scope}[shift={(0,-1.5)}]

\draw (0,0) grid[step=1] (1,3);
\draw (0,0) -- (1,1);
\draw [dashed] (0,3) -- (1,2);
\fill (0.5,0.5) circle (0.2);
\fill (0.5,1.5) circle (0.2);
\fill [gray,opacity=0.7] (1.5,2.5) circle (0.2);

\draw [gray,opacity = 0.7] (1,1) -- (3,3);
\draw [gray,opacity = 0.7] (1,0) grid[step=1] (3,3);

\begin{scope}[shift={(3,6)},rotate=180]
\draw [gray,opacity=0.7] (0,0) grid[step=1] (1,3);
\draw [gray,opacity=0.7] (0,0) -- (1,1);
\draw [gray,opacity=0.7,dashed] (0,3) -- (1,2);
\fill [gray,opacity=0.7] (0.5,0.5) circle (0.2);
\fill [gray,opacity=0.7] (0.5,1.5) circle (0.2);
\fill [gray,opacity=0.7] (1.5,2.5) circle (0.2);

\draw [gray,opacity = 0.7] (1,1) -- (3,3);
\draw [gray,opacity = 0.7] (1,0) grid[step=1] (3,3);
\end{scope}

\end{scope}

\draw (-1,-6) rectangle (4,9);

\begin{scope}[shift={(1,10)}]

\draw (0,0) grid[step=1] (1,3);
\draw (0,0) -- (1,1);
\draw [dashed] (0,3) -- (1,2);
\fill (0.5,0.5) circle (0.2);
\fill (0.5,1.5) circle (0.2);

\end{scope}

\begin{scope}[shift={(5,0)}]

\begin{scope}[shift={(0,2)}]

\draw (0,0) grid[step=1] (1,3);
\draw (0,0) -- (1,1);
\draw [dashed] (0,3) -- (1,2);
\fill (0.5,0.5) circle (0.2);
\fill [gray,opacity=0.7] (1.5,1.5) circle (0.2);
\draw (0.5,2.5) node[scale=0.6]{$\bigstar$};

\draw [gray,opacity = 0.7] (1,1) -- (3,3);
\draw [gray,opacity = 0.7] (1,0) grid[step=1] (3,3);

\begin{scope}[shift={(3,6)},rotate=180]
\draw [gray,opacity=0.7] (0,0) grid[step=1] (1,3);
\draw [gray,opacity=0.7] (0,0) -- (1,1);
\draw [gray,opacity=0.7,dashed] (0,3) -- (1,2);
\fill [gray,opacity=0.7] (0.5,0.5) circle (0.2);
\fill [gray,opacity=0.7] (1.5,1.5) circle (0.2);
\draw [gray,opacity=0.7] (0.5,2.5) node[scale=0.6]{$\bigstar$};

\draw [gray,opacity = 0.7] (1,1) -- (3,3);
\draw [gray,opacity = 0.7] (1,0) grid[step=1] (3,3);
\end{scope}

\end{scope}

\begin{scope}[shift={(0,-5)}]

\draw (0,0) grid[step=1] (1,3);
\draw (0,0) -- (1,1);
\draw [dashed] (0,3) -- (1,2);
\fill (0.5,0.5) circle (0.2);
\fill [gray,opacity=0.7] (1.5,1.5) circle (0.2);
\draw (2.5,2.5) node[scale=0.6]{$\bigstar$};

\draw [gray,opacity = 0.7] (1,1) -- (3,3);
\draw [gray,opacity = 0.7] (1,0) grid[step=1] (3,3);

\begin{scope}[shift={(3,6)},rotate=180]
\draw [gray,opacity=0.7] (0,0) grid[step=1] (1,3);
\draw [gray,opacity=0.7] (0,0) -- (1,1);
\draw [gray,opacity=0.7,dashed] (0,3) -- (1,2);
\fill [gray,opacity=0.7] (0.5,0.5) circle (0.2);
\fill [gray,opacity=0.7] (1.5,1.5) circle (0.2);
\draw [gray,opacity=0.7] (2.5,2.5) node[scale=0.6]{$\bigstar$};

\draw [gray,opacity = 0.7] (1,1) -- (3,3);
\draw [gray,opacity = 0.7] (1,0) grid[step=1] (3,3);
\end{scope}

\end{scope}

\draw (-1,-6) rectangle (4,9);

\begin{scope}[shift={(1,10)}]

\draw (0,0) grid[step=1] (1,3);
\draw (0,0) -- (1,1);
\draw [dashed] (0,3) -- (1,2);
\fill (0.5,0.5) circle (0.2);
\draw (0.5,2.5) node[scale=0.6]{$\bigstar$};

\end{scope}

\end{scope}

\end{tikzpicture}

\caption{Construction of $SDC_3$ from $\mathcal{T}_1^o$.}
\label{fig:SDC3generatedbyT2o}

\end{figure}

\subsubsection{A combinatorial interpretation of $P_n(x)$ in terms of $\mathcal{T}_{n-1}^o$}
We adapt Definition \ref{defi:evenpath}, Definition \ref{defi:TPath}, Proposition \ref{prop:TPathstationnaire} and Definition \ref{defi:green} to the tableaux $T \in \mathcal{T}_n^o$ as follows~:
\begin{itemize}
\item first of all, we replace all the occurrences of $2n$ with $2n+1$, which doesn't change the proof of Proposition \ref{prop:TPathstationnaire};
\item regarding $S_T(i)$ from Definition \ref{defi:evenpath}, if $L_i^T$ is its unique empty row, then $S_T(i)$ is defined as the empty sequence ; otherwise, if the points $p_0,\dots,p_{k_0}$ are defined for some $k_0 \geq~0$, and if the rule~b) or the rule~c) involves the unique empty row of $T$, then we define $p_k$ as $p_{k_0}$ for all $k > k_0$;
\item in Definition \ref{defi:evenpath}, instead of $R(T)$ and $B(T)$, we define a unique finite sequence $V(T) = S_T(2n+1)$ (possibly empty), whose number of elements is denoted by $\text{v}(T)$;
\item the set $\mathcal{G}(T)$ is defined by replacing $B(T)$ with $V(T)$ in Definition~\ref{defi:green}.
\end{itemize}

For example, we represent in Figure~\ref{fig:examplepaths2} a tableau $T_3 \in \mathcal{T}_6^o$ with $V(T_3) = ((2:13),(3:11),(4:10))$ and $\mathcal{G}(T_3) = \{(2:12)\}$. In general, when depicting a tableau $T \in \mathcal{T}_n^o$, the $\text{v}(T)$ elements of $V(T)$ and the $\text{g}(T)$ elements of $\mathcal{G}(T)$ are painted in purple and green respectively.

\begin{figure}[!htbp]

\begin{center}

\begin{tikzpicture}[scale=0.5]
\draw (-0.5,0.5) node[scale=1]{$1$};
\draw (-0.5,1.5) node[scale=1]{$2$};
\draw (-0.5,2.5) node[scale=1]{$3$};
\draw (-0.5,3.5) node[scale=1]{$4$};
\draw (-0.5,4.5) node[scale=1]{$5$};
\draw (-0.5,5.5) node[scale=1]{$6$};
\draw (-0.5,6.5) node[scale=1]{$7$};
\draw (-0.5,7.5) node[scale=1]{$8$};
\draw (-0.5,8.5) node[scale=1]{$9$};
\draw (-0.5,9.5) node[scale=1]{$10$};
\draw (-0.5,10.5) node[scale=1]{$11$};
\draw (-0.5,11.5) node[scale=1]{$12$};
\draw (-0.5,12.5) node[scale=1]{$13$};

\draw (0.5,-0.5) node[scale=1]{$1$};
\draw (1.5,-0.5) node[scale=1]{$2$};
\draw (2.5,-0.5) node[scale=1]{$3$};
\draw (3.5,-0.5) node[scale=1]{$4$};
\draw (4.5,-0.5) node[scale=1]{$5$};
\draw (5.5,-0.5) node[scale=1]{$6$};

\draw (0,0) grid[step=1] (6,13);
\draw (0,0) -- (6,6);
\draw [dashed] (0,13) -- (6,7);
\fill (0.5,0.5) circle (0.2);
\fill (0.5,1.5) circle (0.2);
\fill (2.5,2.5) circle (0.2);
\fill (3.5,3.5) circle (0.2);
\fill (4.5,4.5) circle (0.2);
\fill (4.5,5.5) circle (0.2);
\fill (5.5,6.5) circle (0.2);
\draw (1.5,12.5) node[scale=0.8]{$\bigstar$};
\draw (1.5,11.5) node[scale=0.8]{$\bigstar$};
\draw (2.5,10.5) node[scale=0.8]{$\bigstar$};
\draw (3.5,9.5) node[scale=0.8]{$\bigstar$};
\draw (5.5,7.5) node[scale=0.8]{$\bigstar$};

\draw[color=Purple] (1.5,12.5) node[scale=0.8]{$\bigstar$};
\draw[color=Purple] [very thick] (1.5,12.5) [densely dashed] -- (1.5,10.5);
\draw[color=Purple] [very thick] (1.5,10.5) [densely dashed] -- (2.5,10.5);
\draw[color=Purple] (2.5,10.5) node[scale=0.8]{$\bigstar$};
\draw[color=Purple] [very thick] (2.5,10.5) [densely dashed] -- (2.5,9.5);
\draw[color=Purple] [very thick] (2.5,9.5) [densely dashed] -- (3.5,9.5);
\draw[color=Purple] (3.5,9.5) node[scale=0.8]{$\bigstar$};
\draw[color=Purple] [very thick] (3.5,9.5) [densely dashed] -- (3.5,8.5);

\draw[color=ForestGreen] [very thick] (1.5,1.5) [densely dashed] -- (0.5,1.5);
\draw[color=ForestGreen] [very thick] (0.5,1.5) [densely dashed] --
(0.5,11.5);
\draw[color=ForestGreen] [very thick] (0.5,11.5) [densely dashed] -- (1.5,11.5);
\draw[color=ForestGreen] (1.5,11.5) node[scale=0.8]{$\bigstar$};

\end{tikzpicture}

\end{center}

\caption{Odd extended Dellac configuration $T_3 \in \mathcal{T}_6^o$ such that $(\text{v}(T_3),\text{g}(T_3)) = (\textcolor{Purple}{3},\textcolor{ForestGreen}{1})$.}
\label{fig:examplepaths2}

\end{figure}

\begin{rem}
\label{rem:oddmaximalfreepoints}
For all $T \in \mathcal{T}_{n-1}^o$, the $\text{v}(T)$ elements of $V(T)$ and the $\text{g}(T)$ elements of $\mathcal{G}(T)$ are pairwise distinct free points, so $\text{fr}(T) \geq~\text{v}(T) + \text{g}(T)$.
\end{rem}

The second main result of this paper is the following.

\begin{thm}
\label{theo:PnintermsofTno}
For all $n \geq~2$,
\begin{equation*}
\label{eq:PnintermsofTno}
P_n(x) = \sum_{T \in \mathcal{T}_{n-1}^o} 2^{\text{fr}(T)-\text{g}(T)} x^{\text{v}(T)}(1+x)^{\text{g}(T)}.
\end{equation*}
\end{thm}
For example, the two elements of $\mathcal{T}_1^o$
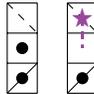
\begin{center}

\begin{tikzpicture}[scale=0.4]

\draw (0,0) grid[step=1] (1,3);
\draw (0,0) -- (1,1);
\draw [dashed] (0,3) -- (1,2);
\fill (0.5,0.5) circle (0.2);
\fill [color=black] (0.5,1.5) circle (0.2);

\begin{scope}[shift={(2,0)}]

\draw (0,0) grid[step=1] (1,3);
\draw (0,0) -- (1,1);
\draw [dashed] (0,3) -- (1,2);
\fill (0.5,0.5) circle (0.2);
\draw[color=Purple] (0.5,2.5) node[scale=0.8]{$\bigstar$};
\draw[color=Purple] [very thick] (0.5,2.5) [densely dashed] -- (0.5,1.5);

\end{scope}

\end{tikzpicture}

\end{center}
give $P_2(x) = 2^{0-\textcolor{ForestGreen}{0}} x^{\textcolor{Purple}{0}} (1+x)^{\textcolor{ForestGreen}{0}} + 2^{1-\textcolor{ForestGreen}{0}} x^{\textcolor{Purple}{1}} (1+x)^{\textcolor{ForestGreen}{0}} = 2x+1$, and the nine elements of $\mathcal{T}_2^o$

\begin{center}

\begin{tikzpicture}[scale=0.4]

\draw (0,0) grid[step=1] (2,5);
\draw (0,0) -- (2,2);
\draw [dashed] (0,5) -- (2,3);
\fill (0.5,0.5) circle (0.2);
\fill (0.5,1.5) circle (0.2);
\fill (1.5,2.5) circle (0.2);
\draw (1.5,3.5) node[scale=0.8]{$\bigstar$};

\begin{scope}[shift={(3,0)}]

\draw (0,0) grid[step=1] (2,5);
\draw (0,0) -- (2,2);
\draw [dashed] (0,5) -- (2,3);
\fill (0.5,0.5) circle (0.2);
\fill (0.5,1.5) circle (0.2);
\fill (1.5,2.5) circle (0.2);
\draw[color=Purple] (1.5,4.5) node[scale=0.8]{$\bigstar$};
\draw[color=Purple] [very thick] (1.65,4.5) [densely dashed] -- (1.65,2.5);
\draw[color=Purple] [very thick] (1.35,2.5) [densely dashed] -- (1.35,1.5);
\draw[color=Purple] [very thick] (1.35,1.5) [densely dashed] -- (0.5,1.5);
\draw[color=Purple] [very thick] (0.5,1.5) [densely dashed] -- (0.5,3.5);

\end{scope}

\begin{scope}[shift={(6,0)}]

\draw (0,0) grid[step=1] (2,5);
\draw (0,0) -- (2,2);
\draw [dashed] (0,5) -- (2,3);
\fill (0.5,0.5) circle (0.2);
\fill (0.5,1.5) circle (0.2);

\draw[color=Purple] (1.5,4.5) node[scale=0.8]{$\bigstar$};
\draw[color=Purple] [very thick] (1.5,4.5) [densely dashed] -- (1.5,2.5);

\draw[color=ForestGreen] (1.5,3.5) node[scale=0.8]{$\bigstar$};
\draw[color=ForestGreen] [very thick] (1.5,3.5) [densely dashed] -- (0.5,3.5);
\draw[color=ForestGreen] [very thick] (0.5,3.5) [densely dashed] -- (0.5,1.5);
\draw[color=ForestGreen] [very thick] (0.5,1.5) [densely dashed] -- (1.5,1.5);

\end{scope}

\begin{scope}[shift={(9,0)}]

\draw (0,0) grid[step=1] (2,5);
\draw (0,0) -- (2,2);
\draw [dashed] (0,5) -- (2,3);
\fill (0.5,0.5) circle (0.2);
\fill (0.5,2.5) circle (0.2);
\fill (1.5,1.5) circle (0.2);
\draw (1.5,3.5) node[scale=0.8]{$\bigstar$};

\begin{scope}[shift={(3,0)}]

\draw (0,0) grid[step=1] (2,5);
\draw (0,0) -- (2,2);
\draw [dashed] (0,5) -- (2,3);
\fill (0.5,0.5) circle (0.2);
\fill (0.5,2.5) circle (0.2);
\fill (1.5,1.5) circle (0.2);
\draw[color=Purple] (1.5,4.5) node[scale=0.8]{$\bigstar$};
\draw[color=Purple] [very thick] (1.5,4.5) [densely dashed] -- (1.5,2.5);
\draw[color=Purple] [very thick] (1.5,2.5) [densely dashed] -- (0.5,2.5);
\draw[color=Purple] [very thick] (0.5,2.5) [densely dashed] -- (0.5,3.5);

\end{scope}

\begin{scope}[shift={(6,0)}]

\draw (0,0) grid[step=1] (2,5);
\draw (0,0) -- (2,2);
\draw [dashed] (0,5) -- (2,3);
\fill (0.5,0.5) circle (0.2);
\fill (1.5,1.5) circle (0.2);
\fill (0.5,3.5) circle (0.2);
\draw[color=Purple] (1.5,4.5) node[scale=0.8]{$\bigstar$};
\draw[color=Purple] [very thick] (1.5,4.5) [densely dashed] -- (1.5,2.5);

\end{scope}

\end{scope}

\begin{scope}[shift={(18,0)}]

\draw (0,0) grid[step=1] (2,5);
\draw (0,0) -- (2,2);
\draw [dashed] (0,5) -- (2,3);
\fill (0.5,0.5) circle (0.2);
\fill (1.5,1.5) circle (0.2);
\fill (0.5,3.5) circle (0.2);
\fill (1.5,2.5) circle (0.2);

\begin{scope}[shift={(3,0)}]

\draw (0,0) grid[step=1] (2,5);
\draw (0,0) -- (2,2);
\draw [dashed] (0,5) -- (2,3);
\fill (0.5,0.5) circle (0.2);
\fill (1.5,1.5) circle (0.2);
\draw[color=Purple] (0.5,4.5) node[scale=0.8]{$\bigstar$};
\draw[color=Purple] [very thick] (0.5,4.5) [densely dashed] -- (0.5,3.5);
\fill (1.5,2.5) circle (0.2);

\end{scope}

\begin{scope}[shift={(6,0)}]

\draw (0,0) grid[step=1] (2,5);
\draw (0,0) -- (2,2);
\draw [dashed] (0,5) -- (2,3);
\fill (0.5,0.5) circle (0.2);
\fill (1.5,1.5) circle (0.2);
\draw[color=Purple] (0.5,4.5) node[scale=0.8]{$\bigstar$};
\draw[color=Purple] [very thick] (0.5,4.5) [densely dashed] -- (0.5,3.5);
\draw[color=Purple] [very thick] (0.5,3.5) [densely dashed] -- (1.5,3.5);
\draw[color=Purple] (1.5,3.5) node[scale=0.8]{$\bigstar$};
\draw[color=Purple] [very thick] (1.5,3.5) [densely dashed] -- (1.5,2.5);
\end{scope}

\end{scope}

\end{tikzpicture}

\end{center}
give
\begin{align*}
P_3(x) = &2^{1-\textcolor{ForestGreen}{0}} x^{\textcolor{Purple}{0}} (1+x)^{\textcolor{ForestGreen}{0}}
+ 2^{1-\textcolor{ForestGreen}{0}} x^{\textcolor{Purple}{1}} (1+x)^{\textcolor{ForestGreen}{0}}
+ 2^{2-\textcolor{ForestGreen}{1}} x^{\textcolor{Purple}{1}} (1+x)^{\textcolor{ForestGreen}{1}}\\
&+ 2^{1-\textcolor{ForestGreen}{0}} x^{\textcolor{Purple}{0}} (1+x)^{\textcolor{ForestGreen}{0}}
+ 2^{1-\textcolor{ForestGreen}{0}} x^{\textcolor{Purple}{1}} (1+x)^{\textcolor{ForestGreen}{0}}
+ 2^{1-\textcolor{ForestGreen}{0}} x^{\textcolor{Purple}{1}} (1+x)^{\textcolor{ForestGreen}{0}}\\
&+ 2^{0-\textcolor{ForestGreen}{0}} x^{\textcolor{Purple}{0}} (1+x)^{\textcolor{ForestGreen}{0}}
+ 2^{1-\textcolor{ForestGreen}{0}} x^{\textcolor{Purple}{1}} (1+x)^{\textcolor{ForestGreen}{0}}
+ 2^{2-\textcolor{ForestGreen}{0}} x^{\textcolor{Purple}{2}} (1+x)^{\textcolor{ForestGreen}{0}}\\
=& 2+2x+2x(1+x)+2+2x+2x+1+2x+4x^2\\
=& 6x^2+10x+5.
\end{align*}

The cardinality of $SDC_{2n-1}$ being $P_n(1) = l_n$ for all $n \geq~2$ is then obtained by setting $x = 1$ in Theorem \ref{theo:PnintermsofTno} in view of formula~\eqref{eq:TnogeneratesSpDC2nmoins1}. Section~\ref{sec:PnintermsofTno} is dedicated to the proof of Theorem \ref{theo:PnintermsofTno}.

\section{Proof of Theorem \ref{theo:PnintermsofTne}}
\label{sec:PnintermsofTne}

For all $n \geq~1$, let
\begin{align*}
E_n(x) = \sum_{T \in \mathcal{T}_n^e} 2^{\text{fr}(T)-1-\max(T)} x^{\max(T)}.
\end{align*}
By Remark \ref{rem:evenmaximalfreepoints} and the inequality $\text{fr}(T) \leq n$ for all $T \in \mathcal{T}_n^e$ (the free points of $T$ are of the kind $(j:i)$ such that $2n \geq~i \geq~2n+1-j \geq~n+1$, hence are located in the $n$ top rows), the polynomial $E_n(x)$ has positive integers coefficients and its degree is at most $n-1$. Moreover, if $T_{n,\max} \in \mathcal{T}_n^e$ is the tableau whose points are $(j:j)$ and $(j:2n+1-j)$ for all $j \in [n]$ (see Figure~\ref{fig:T4max} for the case $n = 4$), then $\text{b}(T_{n,\max}) = 0,\text{r}(T_{n,\max}) = n-1$ and $\text{g}(T_{n,\max}) =0$, so $\max(T_{n,\max}) = n-1$ and the degree of $E_n(x)$ is exactly $n-1$.

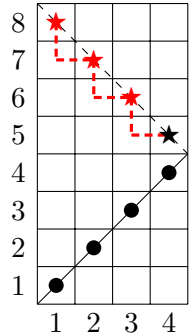
\begin{figure}[!htbp]

\begin{center}

\begin{tikzpicture}[scale=0.5]
\draw (-0.5,0.5) node[scale=1]{$1$};
\draw (-0.5,1.5) node[scale=1]{$2$};
\draw (-0.5,2.5) node[scale=1]{$3$};
\draw (-0.5,3.5) node[scale=1]{$4$};
\draw (-0.5,4.5) node[scale=1]{$5$};
\draw (-0.5,5.5) node[scale=1]{$6$};
\draw (-0.5,6.5) node[scale=1]{$7$};
\draw (-0.5,7.5) node[scale=1]{$8$};

\draw (0.5,-0.5) node[scale=1]{$1$};
\draw (1.5,-0.5) node[scale=1]{$2$};
\draw (2.5,-0.5) node[scale=1]{$3$};
\draw (3.5,-0.5) node[scale=1]{$4$};

\draw (0,0) grid[step=1] (4,8);
\draw (0,0) -- (4,4);
\draw [dashed] (0,8) -- (4,4);
\fill (0.5,0.5) circle (0.2);
\fill (1.5,1.5) circle (0.2);
\fill (2.5,2.5) circle (0.2);
\fill (3.5,3.5) circle (0.2);

\draw[color=red] (0.5,7.5) node[scale=0.8]{$\bigstar$};
\draw[color=red] [very thick] (0.5,7.5) [densely dashed] -- (0.5,6.5);
\draw[color=red] [very thick] (0.5,6.5) [densely dashed] -- (1.5,6.5);
\draw[color=red] (1.5,6.5) node[scale=0.8]{$\bigstar$};
\draw[color=red] [very thick] (1.5,6.5) [densely dashed] -- (1.5,5.5);
\draw[color=red] [very thick] (1.5,5.5) [densely dashed] -- (2.5,5.5);
\draw[color=red] (2.5,5.5) node[scale=0.8]{$\bigstar$};
\draw[color=red] [very thick] (2.5,5.5) [densely dashed] -- (2.5,4.5);
\draw[color=red] [very thick] (2.5,4.5) [densely dashed] -- (3.5,4.5);
\draw[color=black] (3.5,4.5) node[scale=0.8]{$\bigstar$};

\end{tikzpicture}

\end{center}

\caption{The tableau $T_{4,\max} \in \mathcal{T}_4^e$.}
\label{fig:T4max}

\end{figure}

For all $k \in [0,n-1]$, let $\mathcal{T}^e_{n,k}$ be the set of the tableaux $T \in \mathcal{T}_n^e$ such that $\max(T) = k$, and let $e_{n,k} = \sum_{T \in \mathcal{T}_{n,k}^e} 2^{\text{fr}(T)-1-k}$, so that
$$E_n(x) = \sum_{k = 0}^{n-1} e_{n,k} x^k.$$
Theorem~\ref{theo:PnintermsofTne} states that $E_n(x) = P_n(x)$, \textit{i.e.}, that $e_{n,k} = c_{n,k}$ for all $k \in [0,n-1]$. It is clear that $e_{1,0} = 1 = c_{1,0}$. We will prove the rest by checking that $(e_{n,k})_{0 \leq k \leq n}$ satisfies the induction formulas \eqref{eq:an0} ,\eqref{eq:ank},\eqref{eq:anmoins1}. To do so, we first define in the next section a surjection from $\mathcal{T}_n^e$ to $\mathcal{T}_{n-1}^e$.

\subsection{A surjection $\Pi : \mathcal{T}_n^e \twoheadrightarrow \mathcal{T}_{n-1}^e (n \geq~2)$}
Consider $T \in \mathcal{T}_n^e$ and the integers $(j_{n-1},j_{n+1}) \in [n]^2$ such that the points of $L_{n-1}^T$ and $L_{n+1}^T$ are located in $C_{j_{n-1}}^T$ and $C_{j_{n+1}}^T$ respectively. We define $(i_{\min},i_{\max})$ as $(n-1,n+1)$ if $j_{n-1} \leq j_{n+1}$, as $(n+1,n-1)$ otherwise.

We define two new paths
\begin{align*}
(B'(T),R'(T)) = \begin{cases}
(S_T(i_{\min}),S_T(i_{\max})) &\text{if $C_{n-1}^T$ contains an element}\\&\text{of $R(T)$},\\
(S_T(i_{\max}),S_T(i_{\min})) &\text{if $C_{n-1}^T$ contains an element}\\&\text{of $\mathcal{G}(T)$},\\
(S_T(n+1),S_T(n-1)) &\text{otherwise and if $p_{n+1}^T$ is free,}\\
(S_T(n-1),S_T(n+1)) &\text{otherwise},
\end{cases}
\end{align*}
and one new set $\mathcal{G}'(T)$ defined as the free points of $T$ of the kind $p = (j:i)$ such that the other point of $C_j^T$ is an element of $B'(T)$ and $\text{root}_T(j:i) =j$. Their respective numbers of elements located in $C_1^T,C_2^T,\dots,C_{n-1}^T$ are denoted by $\text{b}'(T)$, $\text{r}'(T)$ and $\text{g}'(T)$. Note that $S_T(n+1) = ((n:n+1))$ for all $T$ in which $p_{n+1}^T$ is free, so $r'(T)$ or $b'(T)$ can equal $0$. However, the first element of $S_T(n-1)$, the point $p_{n-1}^T$, is always of the kind $(j:n-1)$ for some $j<n$, so $\text{b}'(T) + \text{r}'(T) \geq~1$ in general. The set of the elements of $B'(T),R'(T)$ or $\mathcal{G}'(T)$ is denoted by $\text{Max}'(T)$ whose cardinality is $\max'(T) = \text{b}'(T) + \text{r}'(T)+\text{g}'(T)$. For example, consider the tableau $T_1 \in \mathcal{T}_7^e$ from Figure~\ref{fig:examplepaths}, we depict in Figure~\ref{fig:newpaths} the paths $B'(T_1) = ((5:6),(6:9))$ and $R'(T_1) = ((6:8))$ by painting them in light blue and orange respectively (there the set $\mathcal{G}(T_1)$ is empty).

\begin{figure}[!htbp]

\begin{center}

\begin{tikzpicture}[scale=0.5]
\draw (-0.5,0.5) node[scale=1]{$1$};
\draw (-0.5,1.5) node[scale=1]{$2$};
\draw (-0.5,2.5) node[scale=1]{$3$};
\draw (-0.5,3.5) node[scale=1]{$4$};
\draw (-0.5,4.5) node[scale=1]{$5$};
\draw (-0.5,5.5) node[scale=1]{$6$};
\draw (-0.5,6.5) node[scale=1]{$7$};
\draw (-0.5,7.5) node[scale=1]{$8$};
\draw (-0.5,8.5) node[scale=1]{$9$};
\draw (-0.5,9.5) node[scale=1]{$10$};
\draw (-0.5,10.5) node[scale=1]{$11$};
\draw (-0.5,11.5) node[scale=1]{$12$};
\draw (-0.5,12.5) node[scale=1]{$13$};
\draw (-0.5,13.5) node[scale=1]{$14$};

\draw (0.5,-0.5) node[scale=1]{$1$};
\draw (1.5,-0.5) node[scale=1]{$2$};
\draw (2.5,-0.5) node[scale=1]{$3$};
\draw (3.5,-0.5) node[scale=1]{$4$};
\draw (4.5,-0.5) node[scale=1]{$5$};
\draw (5.5,-0.5) node[scale=1]{$6$};
\draw (6.5,-0.5) node[scale=1]{$7$};

\draw (0,0) grid[step=1] (7,14);
\draw (0,0) -- (7,7);
\draw [dashed] (0,14) -- (7,7);
\fill (0.5,0.5) circle (0.2);
\fill (0.5,1.5) circle (0.2);
\fill (2.5,2.5) circle (0.2);
\fill (3.5,3.5) circle (0.2);
\fill (4.5,4.5) circle (0.2);

\fill (1.5,6.5) circle (0.2);

\draw (5.5,8.5) node[scale=0.8]{$\bigstar$};
\draw (6.5,9.5) node[scale=0.8]{$\bigstar$};
\draw (6.5,10.5) node[scale=0.8]{$\bigstar$};
\draw (2.5,11.5) node[scale=0.8]{$\bigstar$};
\draw (1.5,12.5) node[scale=0.8]{$\bigstar$};
\draw (3.5,13.5) node[scale=0.8]{$\bigstar$};

\draw[color=red] (3.5,13.5) node[scale=0.8]{$\bigstar$};
\draw[color=red] [very thick] (3.5,13.5) [densely dashed] -- (3.5,9.5);
\draw[color=red] [very thick] (3.5,9.5) [densely dashed] -- (6.5,9.5);

\fill[color=Peach] (5.5,7.5) circle (0.2);

\fill[color=blue] (1.5,6.5) circle (0.2);
\draw[color=blue] [very thick] (1.5,6.5) [densely dashed] -- (1.5,1.5);
\draw[color=blue] [very thick] (1.5,1.5) [densely dashed] -- (0.5,1.5);
\draw[color=blue] [very thick] (0.5,1.5) [densely dashed] -- (0.5,12.5);
\draw[color=blue] [very thick] (0.5,12.5) [densely dashed] -- (1.5,12.5);
\draw[color=blue] [very thick] (1.5,12.5) [densely dashed] -- (1.5,11.5);
\draw[color=blue] [very thick] (1.5,11.5) [densely dashed] -- (2.5,11.5);
\draw[color=blue] (2.5,11.5) node[scale=0.8]{$\bigstar$};
\draw[color=blue] [very thick] (2.5,11.5) [densely dashed] -- (2.5,10.5);
\draw[color=blue] [very thick] (2.5,10.5) [densely dashed] -- (6.5,10.5);

\fill[color=Turquoise] (4.5,5.5) circle (0.2);
\draw[color=Turquoise] [very thick] (4.5,5.5) [densely dashed] -- (4.5,8.5);
\draw[color=Turquoise] [very thick] (4.5,8.5) [densely dashed] -- (5.5,8.5);
\draw[color=Turquoise] (5.5,8.5) node[scale=0.8]{$\bigstar$};

\draw[color=ForestGreen] [very thick] (1.3,1.3) [densely dashed] -- (0.3,1.3);
\draw[color=ForestGreen] [very thick] (0.3,1.3) [densely dashed] -- (0.3,12.7);
\draw[color=ForestGreen] [very thick] (0.3,12.7) [densely dashed] -- (1.3,12.7);
\draw[color=ForestGreen] (1.5,12.5) node[scale=0.8]{$\bigstar$};

\end{tikzpicture}

\end{center}

\caption{The tableau $T_1 \in \mathcal{T}_7^e$, such that $(\text{b}'(T_1),\text{r}'(T_1),\text{g}'(T_1)) = (\textcolor{Turquoise}{2},\textcolor{Peach}{1},\textcolor{YellowGreen}{0})$.}
\label{fig:newpaths}

\end{figure}
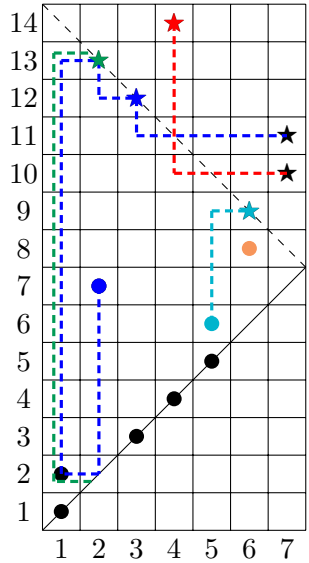

In order to construct an element $\Pi(T)$ of $\mathcal{T}_{n-1}^e$ from $T \in \mathcal{T}_n^e$, we "project" all points of $T$ that are not elements of $\text{Max}(T)$ or $\text{Max}'(T)$  in a tableau $(n-1)\times(2n-2)$, then we construct $\text{Max}(\Pi(T))$ as a "union" of $\text{Max}(T)$ and $\text{Max}'(T)$, in a meaning that we will make explicit.

\begin{dfn}
\label{defi:insertion}
Let $j \in [n]$ and consider a tableau $T$ satisfying the conditions of Definition~\ref{defi:TPath}. Let $i_f \in [j,2n-j] \sqcup \{2n\}$ and let $i_0 \in [j,2n]$ be the unique integer such that $\text{root}_T(j:i_0) = i_f$ in view of Proposition \ref{prop:TPathstationnaire}. We define the insertion of a point in the box $(j:i_f)$ of $T$ as the plotting of a point in the box $(j:i_0)$.

Note that if $L_{i_f}^T$ has its $j-1$ first boxes empty, then $i_0 = i_f$.
\end{dfn}
For example, if $T$ is a $7 \times 14$ tableau whose first column is the same as in Figure~\ref{fig:newpaths}, then the insertions of points in the boxes $(2:2)$ and $(2:7)$ of $T$ result in plotting points in the boxes $(2:13)$ and $(2:7)$ of $T$ respectively.

In general, let us now label the points of $\text{Max}(T)$ and $\text{Max}'(T)$ with letters $\rho,\beta$ or $\gamma$. The points labeled with $\beta$ (respectively $\rho,\gamma$) will correspond with the elements of $B(\Pi(T))$ (respectively $R(\Pi(T)), \mathcal{G}(\Pi(T))$).

\begin{dfn}
\label{defi:labels}
Let $T \in \mathcal{T}_n^e (n \geq~2)$. We label the elements of $\text{Max}(T)$ and $\text{Max}'(T)$ as follows.
\begin{itemize}
\item If a point $p_1$ belongs to $B(T)$ (respectively $R'(T)$) and the other point $p_2$ in the same column belongs to $B'(T)$ (respectively $R(T)$), then we label $p_1$ with $\beta$ and $p_2$ with $\gamma$.
\item If a point $p$ belongs to $B(T)$ or $B'(T)$ (respectively $R(T)$ or $R'(T)$, $\mathcal{G}(T)$ or $\mathcal{G}'(T)$) and is not concerned by the first rule, then we label it with $\beta$ (respectively $\rho,\gamma$).
\end{itemize}
\end{dfn}

\begin{dfn}[Map $\Pi : \mathcal{T}_n^e \rightarrow \mathcal{T}_{n-1}^e$]
\label{defi:Pi}
Let $T \in \mathcal{T}_n^e$. We first define a rectangular tableau $\Pi_0(T)$, made of $n-1$ columns and $2n-2$ rows, that contains a point in its box $\left(j:\widetilde{i}\right)$ for all points $p = (j:i)$ of $T$ that doesn't belong to $\text{Max}(T)$ or $\text{Max}'(T)$ (in particular $j < n$ and $i \not\in\{n,n+1\}$), where
$$\widetilde{i} = \begin{cases}
i &\text{if $i <n$},\\
i-2 &\text{if $i>n+1$.}
\end{cases}$$
For the example of $T_1 \in \mathcal{T}_7^e$ depicted in Figure~\ref{fig:newpaths}, we obtain the tableau $\Pi_0(T_1)$ of Figure~\ref{fig:Pi0T1}.

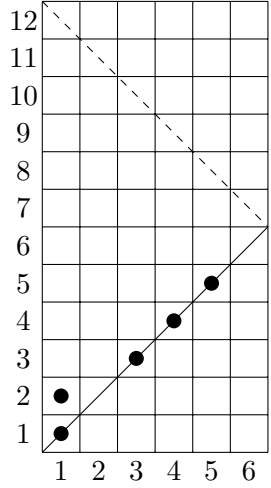
\begin{figure}[!htbp]
\begin{center}

\begin{tikzpicture}[scale=0.5]
\draw (-0.5,0.5) node[scale=1]{$1$};
\draw (-0.5,1.5) node[scale=1]{$2$};
\draw (-0.5,2.5) node[scale=1]{$3$};
\draw (-0.5,3.5) node[scale=1]{$4$};
\draw (-0.5,4.5) node[scale=1]{$5$};
\draw (-0.5,5.5) node[scale=1]{$6$};
\draw (-0.5,6.5) node[scale=1]{$7$};
\draw (-0.5,7.5) node[scale=1]{$8$};
\draw (-0.5,8.5) node[scale=1]{$9$};
\draw (-0.5,9.5) node[scale=1]{$10$};
\draw (-0.5,10.5) node[scale=1]{$11$};
\draw (-0.5,11.5) node[scale=1]{$12$};

\draw (0.5,-0.5) node[scale=1]{$1$};
\draw (1.5,-0.5) node[scale=1]{$2$};
\draw (2.5,-0.5) node[scale=1]{$3$};
\draw (3.5,-0.5) node[scale=1]{$4$};
\draw (4.5,-0.5) node[scale=1]{$5$};
\draw (5.5,-0.5) node[scale=1]{$6$};

\draw (0,0) grid[step=1] (6,12);
\draw (0,0) -- (6,6);
\draw [dashed] (0,12) -- (6,6);
\fill (0.5,0.5) circle (0.2);
\fill (0.5,1.5) circle (0.2);
\fill (2.5,2.5) circle (0.2);
\fill (3.5,3.5) circle (0.2);
\fill (4.5,4.5) circle (0.2);

\end{tikzpicture}

\end{center}
\caption{The tableau $\Pi_0(T_1)$.}
\label{fig:Pi0T1}

\end{figure}
Afterwards, let $X_T$ be the empty set. For $j$ from $1$ to $n-1$, if $C_j^T$ contains a point $p$ labeled with $\beta$ (respectively $\rho,\gamma$), we insert a point in the box $(j:n-1)$ of $\Pi_0(T)$ (respectively the box $(j:2n-2)$, the box $(j:j)$), following Definition~\ref{defi:insertion}. We say that \textit{we have inserted $p$ in $\Pi_0(T)$}. Also, if $p \in \text{Max}(T)$, let $p_{\text{new}}$ be the point just plotted in $\Pi_0(T)$ following the latter insertion, then we add $p_{\text{new}}$ to $X_T$.

Let $\Pi(T)$ be the tableau obtained at the end of this algorithm. It is straightforward by Definition \ref{defi:insertion} that $\Pi(T)$ is an element of $\mathcal{T}_{n-1}^e$, and that $X_T \subsetneq \text{Omax}(\Pi(T))$ (it is not equal because the point $p_{n-1}^{\Pi(T)}$ is an element of $\text{Omax}(\Pi(T))$ that results from the insertion in $\Pi_0(T)$ of the point $p_{n-1}^T$, but $p_{n-1}^T$ belongs to $\text{Max}'(T)$, not to $\text{Max}(T)$, hence $p_{n-1}^{\Pi(T)} \not\in X_T$). The set $X_T$ will be useful to reconstruct $T$ from $\Pi(T)$ (see Proposition \ref{prop:Pimoins1}).

For example, the tableau $\Pi(T_1) \in \mathcal{T}_6^e$ is as depicted in Figure~\ref{fig:PiT1}, and $X_{T_1} = \{(2:6),(2:11),(3:10),(4:12)\}$.

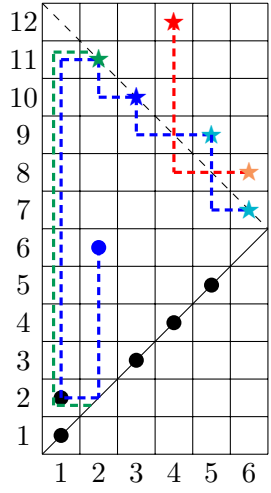
\begin{figure}[!htbp]
\begin{center}

\begin{tikzpicture}[scale=0.5]
\draw (-0.5,0.5) node[scale=1]{$1$};
\draw (-0.5,1.5) node[scale=1]{$2$};
\draw (-0.5,2.5) node[scale=1]{$3$};
\draw (-0.5,3.5) node[scale=1]{$4$};
\draw (-0.5,4.5) node[scale=1]{$5$};
\draw (-0.5,5.5) node[scale=1]{$6$};
\draw (-0.5,6.5) node[scale=1]{$7$};
\draw (-0.5,7.5) node[scale=1]{$8$};
\draw (-0.5,8.5) node[scale=1]{$9$};
\draw (-0.5,9.5) node[scale=1]{$10$};
\draw (-0.5,10.5) node[scale=1]{$11$};
\draw (-0.5,11.5) node[scale=1]{$12$};

\draw (0.5,-0.5) node[scale=1]{$1$};
\draw (1.5,-0.5) node[scale=1]{$2$};
\draw (2.5,-0.5) node[scale=1]{$3$};
\draw (3.5,-0.5) node[scale=1]{$4$};
\draw (4.5,-0.5) node[scale=1]{$5$};
\draw (5.5,-0.5) node[scale=1]{$6$};

\draw (0,0) grid[step=1] (6,12);
\draw (0,0) -- (6,6);
\draw [dashed] (0,12) -- (6,6);
\fill (0.5,0.5) circle (0.2);
\fill (0.5,1.5) circle (0.2);
\fill (2.5,2.5) circle (0.2);
\fill (3.5,3.5) circle (0.2);
\fill (4.5,4.5) circle (0.2);

\draw[color=red] (3.5,11.5) node[scale=0.8]{$\bigstar$};
\draw[color=red] [very thick] (3.5,11.5) [densely dashed] -- (3.5,7.5);
\draw[color=red] [very thick] (3.5,7.5) [densely dashed] -- (5.5,7.5);
\draw[color=Peach] (5.5,7.5) node[scale=0.8]{$\bigstar$};

\fill[color=blue] (1.5,5.5) circle (0.2);
\draw[color=blue] [very thick] (1.5,5.5) [densely dashed] -- (1.5,1.5);
\draw[color=blue] [very thick] (1.5,1.5) [densely dashed] -- (0.5,1.5);
\draw[color=blue] [very thick] (0.5,1.5) [densely dashed] -- (0.5,10.5);
\draw[color=blue] [very thick] (0.5,10.5) [densely dashed] -- (1.5,10.5);
\draw[color=blue] [very thick] (1.5,10.5) [densely dashed] -- (1.5,9.5);
\draw[color=blue] [very thick] (1.5,9.5) [densely dashed] -- (2.5,9.5);
\draw[color=blue] (2.5,9.5) node[scale=0.8]{$\bigstar$};
\draw[color=blue] [very thick] (2.5,9.5) [densely dashed] -- (2.5,8.5);
\draw[color=blue] [very thick] (2.5,8.5) [densely dashed] -- (4.5,8.5);
\draw[color=Turquoise] (4.5,8.5) node[scale=0.8]{$\bigstar$};
\draw[color=blue] [very thick] (4.5,8.5) [densely dashed] -- (4.5,6.5);
\draw[color=blue] [very thick] (4.5,6.5) [densely dashed] -- (5.5,6.5);
\draw[color=Turquoise] (5.5,6.5) node[scale=0.8]{$\bigstar$};

\draw[color=ForestGreen] [very thick] (1.3,1.3) [densely dashed] -- (0.3,1.3);
\draw[color=ForestGreen] [very thick] (0.3,1.3) [densely dashed] -- (0.3,10.7);
\draw[color=ForestGreen] [very thick] (0.3,10.7) [densely dashed] -- (1.3,10.7);
\draw[color=ForestGreen] (1.5,10.5) node[scale=0.8]{$\bigstar$};

\end{tikzpicture}

\end{center}
\caption{The tableau $\Pi(T_1) \in \mathcal{T}_6^e$.}
\label{fig:PiT1}
\end{figure}

Note that $(\text{b}(\Pi(T_1)),\text{r}(\Pi(T_1)),\text{g}(\Pi(T_1))) = (3,1,1)$.
\end{dfn}

\begin{rem}
\label{rem:newstat}
For all $T \in \mathcal{T}_n^e$, we have
\begin{align*}
\max(\Pi(T)) = \max(T) + \text{max}'(T)-2.
\end{align*}
Indeed, the insertions in $\Pi_0(T)$ of the two points of $C_{n-1}^T$, which both belong to $\text{Max}(T)$ or $\text{Max}'(T)$ (because their roots belong to $\{n-1,n,n+1,2n\}$), produce the two points of $C_{n-1}^{\Pi(T)}$, in other words, the last element of $B(\Pi(T))$ and the last element of $R(\Pi(T))$, which are not counted by the statistics $\text{b}$ and $\text{r}$.

In particular, since $\max'(T) \geq~1$ in general (the set $\text{Max}'(T)$ always contains $p_{n-1}^T$), we have $$\max(\Pi(T)) \in [\max(T)-1,n-2].$$

For the example of $T_1 \in \mathcal{T}_7^e$, we obtain $\max(\Pi(T_1)) = (2+1+1)+(2+1)-2 = 5$.
\end{rem}

We now intend to construct $\Pi^{-1}(T_0)$ for all $T_0 \in \mathcal{T}_{n-1}^e$. To do so, we first need to define labels for the elements of $\text{Omax}(T_0)$, given a subset $X \subsetneq \Omax(T_0)$.

\begin{dfn}
\label{defi:newlabels}
For a tableau $T_0 \in \mathcal{T}_{n-1}^e$ and a subset $X \subsetneq \text{Omax}(T_0)$, let us define labels of the elements of $\text{Omax}(T_0)$ as follows.
\begin{itemize}
\item For a point $p \in \mathcal{G}(T_0)$, let $q \neq p$ be the other point of the column of $p$ (the point $q$ is then an element of $B(T_0)$).
\begin{itemize}
\item If neither $p$ nor $q$ belong to $X$, we label $p$ and $q$ with the respective letters $\mathfrak{g}'$ and $\mathfrak{b}'$.
\item If only $q \in X$, we label $p$ and $q$ with the respective letters $\mathfrak{b}'$ and $\mathfrak{b}$.
\item If only $p \in X$, we label $p$ and $q$ with the respective letters $\mathfrak{r}$ and $\mathfrak{r}'$.
\item If both $p$ and $q$ belong to $X$, we label $p$ and $q$ with the respective letters $\mathfrak{g}$ and $\mathfrak{b}$.
\end{itemize}
\item For all points $p$ among the $b_0(T)$ first elements of $B(T_0)$ (respectively among the $r(T_0)$ first elements of $R(T_0)$) that is not concerned by the first case, if $p \not\in X$, then we label $p$ with $\mathfrak{b}'$ (respectively $\mathfrak{r}'$), otherwise we label it with $\mathfrak{b}$ (respectively $\mathfrak{r}$).
\item Let $p^b$ (respectively $p^r$) be the last element of $B(T_0)$ (respectively $R(T_0)$). Recall that $p^b$ and $p^r$ are the two points of $C_{n-1}^{T_0}$.
\begin{itemize}
\item  If $p^r \not\in X$, we label $p^r$ with the letter $\mathfrak{r}'$. Afterwards, we label $p^b$ with the letter $\mathfrak{b}'$ if $p^b \not\in X$, with the letter $\mathfrak{b}$ otherwise.
\item Otherwise, let $j_{\min}$ be the smallest integer $j \in [n-1]$ such that $C_j^{T_0}$ contains an element of $\text{Omax}(T_0) \backslash X$. If $j_{\min} = n-1$, then $\Omax(T_0) = \{p^b,p^r\}$, hence $X = \{p^r\}$, and we label $p^b$ and $p^r$ with the respective letters $\mathfrak{b}'$ and $\mathfrak{r}$. If $j_{\min} < n-1$, all the elements of $\Omax(T_0)$ that appear in $C_{j_{\min}}^{T_0}$ have already been labeled. We then have two cases.
\begin{itemize}
\item If $C_{j_{\min}}^{T_0}$ contains an element labeled with $\mathfrak{b}'$, then we label $p^r$ with the letter $\mathfrak{r}$. Afterwards, we label $p^b$ with the letter $\mathfrak{b}'$ if $p^b \not\in X$, with the letter $\mathfrak{b}$ otherwise.
\item Otherwise, we label $p^b$ and $p^r$ with the respective letters $\mathfrak{r}'$ and $\mathfrak{r}$ if $p^b \not\in X$, with the respective letters $\mathfrak{b}$ and $\mathfrak{g}$ otherwise.
\end{itemize}
\end{itemize}
\end{itemize}
\end{dfn}

\begin{rem}
\label{rem:labelsPi}
In the context of Definition \ref{defi:newlabels}, the elements of $\text{Omax}(T_0)$ that are labeled with $\mathfrak{b},\mathfrak{r}$ or $\mathfrak{g}$ are exactly the elements of $X$.
\end{rem}

\begin{dfn}
\label{defi:Pimoins1}
Let $i \in [0,n-2]$, and $T_0 \in \mathcal{T}_{n-1,i}^e$. We consider a subset $X \subsetneq \text{Omax}(T_0)$, whose cardinality $|X|$ is denoted by $k$. Note that $\text{omax}(T_0) = i+2$ hence $k \in [0,i+1]$. We intend to define a set $\mathfrak{T}(T_0,X) \subset \mathcal{T}_{n}^e$ whose cardinality is either $1$ or $2$ following three situations described below. Consider the labels introduced in Definition \ref{defi:newlabels}.
\begin{enumerate}[label=\arabic*.]
\item  Let $\mathcal{S}_1$ be the situation where $k \leq i$ and no point of the last column $C_{n-1}^{T_0}$ of $T_0$ is labeled with $\mathfrak{r}$ or $\mathfrak{g}$.
\item Let $\mathcal{S}_2$ be the situation where~:
\begin{itemize}
\item either $k = i+1$;
\item or one of the points of $C_{n-1}^{T_0}$ is labeled with $\mathfrak{r}$ or $\mathfrak{g}$, and there are no two elements of $\Omax(T_0) \backslash X$ labeled with $\mathfrak{b}'$ and $\mathfrak{r}'$.
\end{itemize}
\item Let $\mathcal{S}_3$ be the remaining situation, \textit{i.e.}, where $k \leq i$, one of the points of $C_{n-1}^{T_0}$ is labeled with $\mathfrak{r}$ or $\mathfrak{g}$, and there exist two elements of $\Omax(T_0) \backslash X$ labeled with $\mathfrak{b}'$ and $\mathfrak{r}'$.
\end{enumerate}
In any situation, first consider a $n \times 2n$ tableau $\mathfrak{T}_0(T_0)$ that contains points only in its boxes $\left( j:\overline{i} \right)$ for all points $(j:i)$ of $T_0$ that doesn't belong to $\text{Omax}(T_0)$, where $\overline{i}$ is defined as $i$ if $i < n$, as $i+2$ otherwise. Let us set $T^1 = T^2 = T^{3/} = T^{3\backslash} = \mathfrak{T}_0(T_0)$. For $j$ from $1$ to $n-1$, assume that $C_j^{T_0}$ contains an element $p \in \text{Omax}(T_0)$ (otherwise $T^1,T^2,T^{3/},T^{3\backslash}$ already contain two points in their $j$-th columns and we do nothing).
\begin{itemize}
\item If the label of $p$ is $\mathfrak{b}$ (respectively $\mathfrak{r},\mathfrak{g},\mathfrak{g}'$), we insert a point in the box $(j:n)$ (respectively $(j:2n),(j:j),(j:j)$) of $T^1,T^2,T^{3/}$ and $T^{3\backslash}$.
\item If the label of $p$ is $\mathfrak{b}'$ (respectively $\mathfrak{r}'$),
\begin{itemize}
\item in the situation $\mathcal{S}_1$, we insert a point in the box $(j:n-1)$ (respectively $(j:n+1)$) of $T^1$;
\item in the situation $\mathcal{S}_2$, we insert a point in the box $(j:n-1)$ of $T^2$;
\item in the situation $\mathcal{S}_3$,
\begin{itemize}
\item if the points of the last column of $T_0$ have the labels $\mathfrak{r}'$ and $\mathfrak{r}$, or $\mathfrak{b}$ and $\mathfrak{g}$, then we insert a point in the box $(j:n+1)$ (respectively $(j:n-1)$) of $T^{3/}$, and we insert a point in the box $(j:n-1)$ (respectively $(j:n+1)$) of $T^{3\backslash}$;
\item otherwise, we insert a point in the box $(j:n-1)$ (respectively $(j:n+1)$) of $T^{3/}$, and we insert a point in the box $(j:n+1)$ (respectively $(j:n-1)$) of $T^{3\backslash}$.
\end{itemize}
\end{itemize}
\end{itemize}
We say that \textit{we inserted $p$ in $T^1$} (respectively $T^2$, $T^{3/}$ and $T^{3 \backslash}$) in the situation $\mathcal{S}_1$ (respectively $\mathcal{S}_2,\mathcal{S}_3$). At the end of this algorithm, we finally insert two points in the boxes $(n:n)$ and $(n:2n)$ of $T^1$ (respectively $T^2$, $T^{3/}$ and $T^{3 \backslash}$) in the situation $\mathcal{S}_1$ (respectively $\mathcal{S}_2,\mathcal{S}_3$).

The choice of $T_0$ and $X \subsetneq \text{Omax}(T_0)$ raises one of the three situations $\mathcal{S}_1,\mathcal{S}_2,\mathcal{S}_3$, following which either $T^1$, or $T^2$, or both $T^{3/}$ and $T^{3\backslash}$ belong to $\mathcal{T}_{n}^e$ at the end of the above algorithm. Let $\mathfrak{T}(T_0,X) \subset \mathcal{T}_n^e$ be defined as $\{T^1\}$ (respectively $\{T^2\},\{T^{3/},T^{3 \backslash}\}$) in the situation $\mathcal{S}_1$ (respectively $\mathcal{S}_2,\mathcal{S}_3$).
\end{dfn}

\begin{rem}
\label{rem:maximalPimoins1}
If $T \in \mathfrak{T}(T_0,X)$, then by construction the cardinality of $\text{Max}(T)$ is the number of points of $\text{Omax}(T_0)$ labeled with $\mathfrak{b},\mathfrak{r}$ or $\mathfrak{g}$, in other words the cardinality $k$ of $X$ by Remark \ref{rem:labelsPi}, and $\mathfrak{T}(T_0,X) \subset \mathcal{T}_{n,k}^e$.
\end{rem}

\begin{rem}
\label{rem:freepointsofantecedents}
In the context of Definition \ref{defi:Pimoins1}, if a tableau is defined in the situation $\mathcal{S}_1$ or $\mathcal{S}_3$, then the point $p_{n+1}^T$ is not free, in other words the box $(n:n+1)$ of $T$ is empty, whereas in the situation $\mathcal{S}_2$ the point $p_{n+1}^T$ is $(n:n+1)$ hence is free. Indeed, in the situation $\mathcal{S}_1$, the points of $C_{n-1}^{T_0}$ have the labels $\mathfrak{b}'$ and $\mathfrak{r}'$, so the points of $C_{n-1}^T$ are obtained by inserting two points in the boxes $(n-1:n-1)$ and $(n-1:n+1)$ of $T^1$, which implies that $p_{n+1}^T$ is located in the $n-1$ first columns of $T$ and $p_{n+1}^T$ is then not free ; in the situation $\mathcal{S}_3$, since there exists two elements of $\Omax(T_0) \backslash X$ with the labels $\mathfrak{b}'$ and $\mathfrak{r}'$, by construction we insert a point in boxes $(j_1:n-1)$ and $(j_2:n+1)$ for some $j_1,j_2<n$ in both tableaux $T^{3/}$ and $T^{3 \backslash}$, so $p_{n+1}^T$ is also not free ; finally, in the situation $\mathcal{S}_2$, we never insert a point in any box $(j:n+1)$ for all $j <n$, so $p_{n+1}^T$ is necessarily located in $C_n^T$ hence is free.

Let us now compare the number of free points in $T_0$ and in the elements of $\mathfrak{T}(T_0,X)$. First of all, note that the number of free points of $T_0$ is the number of free points in $\mathfrak{T}_0(T_0)$ (let's denote it by $\text{fr}_0(T_0)$), plus $\text{omax}(T_0)-1 = i+1$ in view of Remark \ref{rem:evenmaximalfreepoints}. Likewise, if $T \in \mathfrak{T}(T_0,X)$, by construction $$\text{fr}(T) = \text{fr}_0(T_0) + (\text{max}(T)+1) + \text{max}'(T) - \eta$$ where $\eta$ is the number of non-free points of $\text{Max}'(T)$, \textit{i.e.}, $\eta = 1$ if $p_{n+1}^T$ is free (in which case $p_{n-1}^T$ is the only non-free point of $\text{Max}'(T)$), which happens exactly in the situation $\mathcal{S}_2$, or $\eta = 2$ otherwise, in the situations $\mathcal{S}_1$ and $\mathcal{S}_3$. Since $\text{max}(T) = k$ (in view of Remark \ref{rem:maximalPimoins1}) and $\max'(T) = \text{omax}(T_0) - |X| = i+2-k$ by construction, we then have
$$\begin{cases}
\text{fr}(T) = \text{fr}_0(T_0) + i+3 - 2 = \text{fr}(T_0) &\text{in the situation $\mathcal{S}_1$ or $\mathcal{S}_3$},\\
\text{fr}(T) = \text{fr}_0(T_0) + i+3 - 1 = \text{fr}(T_0)+1 &\text{in the situation $\mathcal{S}_2$.}
\end{cases}$$
\end{rem}

\begin{prop}
\label{prop:Pimoins1}
For all $i \in [0,n-2]$ and $T_0 \in \mathcal{T}_{n-1,i}$, we have $\Pi^{-1}(T_0) \subset \bigsqcup_{k = 0}^{i+1} \mathcal{T}_{n,k}^e$, and for all $k \in [0,i+1]$,
\begin{equation}
\label{eq:Pimoins1}
\Pi^{-1}(T_0) \cap \mathcal{T}_{n,k}^e  = \bigsqcup_{X \subset \text{Omax}(T_0),|X|=k} \mathfrak{T}(T_0,X).
\end{equation}
Therefore, $\Pi$ is surjective.
\end{prop}

\begin{proof}
Let $T \in \mathcal{T}_n^e$, if $\Pi(T) = T_0$, then $i \in [\max(T)-1,n-2]$ by Remark~\ref{rem:evenmaximalfreepoints}, so $\Pi^{-1}(T_0) \subset \bigsqcup_{k = 0}^{i+1} \mathcal{T}_{n,k}^e$.

Let us now prove that the elements $T \in \mathfrak{T}(T_0,X)$ for all $X \subsetneq \text{Omax}(T_0)$ are indeed mapped to $T_0$ by $\Pi$. It is straightforward that $\Pi_0(T)$ is obtained by erasing all the elements of $\text{Omax}(T_0)$ from $T_0$. Afterwards, for a point $p \in \text{Omax}(T_0)$, let $\bar{p} \in \text{Max}(T) \sqcup \text{Max}'(T)$ be the point of $T$ obtained by inserting $p$ in $\mathfrak{T}_0(T_0)$. One can check that the label of $\bar{p}$ as defined by Definition \ref{defi:labels} is $\beta$ (respectively $\rho,\gamma$) if and only if $p \in B(T_0)$ (respectively $p \in R(T_0),\mathcal{G}(T_0)$), so we obtain $\Pi(T) = T_0$
and $\mathfrak{T}(T_0,X)\subset \Pi^{-1}(T_0)$.

Reciprocally, for all $T \in \Pi^{-1}(T_0)$, we show that $T \in \mathfrak{T}(T_0,X_T)$. It is straightforward that $\mathfrak{T}_0(T_0)$ is obtained by erasing all the elements of $\text{Omax}(T)$ and $\text{Max}'(T)$ from $T$. Afterwards, for a point $p \in \text{Max}(T) \sqcup \text{Max}'(T)$, let $\tilde{p} \in \text{Omax}(T_0)$ be the point of $T_0$ that was obtained by inserting $p$ in in $\Pi_0(T)$. One can check that if $p \in B(T)$ (respectively $p \in R(T),\mathcal{G}(T),B'(T),R'(T),\mathcal{G}'(T)$), then the label of $\tilde{p}$ as defined by Definition \ref{defi:newlabels}  with $X = X_T$ (recall that $X_T$ is defined at the same time as $\Pi(T)$ in Definition \ref{defi:Pi}) is $\mathfrak{b}$ (respectively $\mathfrak{r},\mathfrak{g},\mathfrak{b}',\mathfrak{r}',\mathfrak{g}'$). Consequently, if $\mathfrak{T}(T_0,X_T)$ is defined in the situation $\mathcal{S}_1$ or $\mathcal{S}_2$, then $ \mathfrak{T}(T_0,X_T) = \{T\}$. If it is in the situation $\mathcal{S}_3$, then $\mathfrak{T}(T_0,X_T) = \{T^{3/},T^{3 \backslash}\}$ where $T = T^{3/}$ if $i_{\min}^T = n-1$, otherwise $T = T^{3 \backslash}$.

So $\Pi^{-1}(T_0) = \bigsqcup_{X \subsetneq \Omax(T_0)} \mathfrak{T}(T_0,X)$. formula~\eqref{eq:Pimoins1} then follows from Remark \ref{rem:maximalPimoins1}.
\end{proof}

\begin{lem}
\label{lem:S1}
Let $T_0 \in \mathcal{T}_{n-1}^e,X\subsetneq \Omax(T_0)$ and $T \in \mathfrak{T}(T_0,X)$. The following assertions are equivalent.
\begin{enumerate}
\item The set $\mathfrak{T}(T_0,X)$ is defined in the situation $\mathcal{S}_1$ of Definition \ref{defi:Pimoins1}.
\item The point $p_{n+1}^T$ is not free and no point of $C_{n-1}^T$ is an element of $R(T)$ or $\mathcal{G}(T)$.
\item The set $X$ doesn't contain the last element $p^r$ of $R(T_0)$, and $|X| \leq \max(T_0)$.
\end{enumerate}
\end{lem}

\begin{proof}
The equivalence (i) $\Leftrightarrow$ (iii) is straightforward by the third point of Definition \ref{defi:newlabels}.
Suppose that $\mathfrak{T}(T_0,X)$ is defined in the situation $\mathcal{S}_1$ of Definition \ref{defi:Pimoins1}, hence $\mathfrak{T}(T_0,X) = \{T\}$. We know that $\max(T) = |X| \leq \max(T_0)$ by hypothesis, so $\max'(T) = \omax(T_0) - |X| \geq~2$. Now, as stated in the proof of Proposition \ref{prop:Pimoins1}, the hypothesis that no point of $C_{n-1}^{T_0}$ has the label $\mathfrak{r}$ or $\mathfrak{g}$ implies that no point of $C_{n-1}^T$ belongs to $R(T)$ or $\mathcal{G}(T)$. It remains to prove that $p_{n+1}^T$ is free. In general, by Proposition \ref{prop:TPathstationnaire}, the roots of the points of $C_{n-1}^T$ belong to $\{n-1,n,n+1,2n\}$. Since no point of $C_{n-1}^T$ belongs to $R(T)$ by hypothesis, their roots cannot be $2n$. If a point $p$ of $C_{n-1}^T$ has the root $n+1$, then obviously $p_{n+1}^T$ is located in one of the $n-1$ first columns of $T$, so it is not free. Otherwise, one point $p \in C_{n-1}^T$ has the root $n-1$ and the other point $q$ of $C_{n-1}^T$ has the root $n$. It implies that $q$ is an element of $B(T)$. Since $p \not\in \mathcal{G}(T)$ by hypothesis, it cannot be free in this situation, so $p = (n-1:n-1)$ and $S_T(n-1) = (p)$. Since $\max'(T) \geq~2$ by hypothesis, the sequence $S_T(n+1)$ has at least one element located in the $n-1$ first columns of $T$, in other words $p_{n+1}^T$ (the first element of $S_T(n+1)$) is not free. So (i) implies (ii).

Let us now prove that (ii) implies (i). If $p_{n+1}^T$ is not free and no point of $C_{n-1}^T$ belongs to $R(T)$ or $\mathcal{G}(T)$, as stated in the proof of Proposition \ref{prop:Pimoins1} no point of $C_{n-1}^T$ has the label $\mathfrak{r}$ or $\mathfrak{g}$. Also, since $\max'(T) \geq~2$ because $\Max(T)$ contains $p_{n-1}^T$ in general and $p_{n+1}^T$ in this case (because $p_{n+1}^T$ is not free), we have $|X| = \max(T) = \omax(T_0) - \max'(T) = \max(T_0)+2-\max'(T) \leq \max(T_0)$. So (ii) implies (i).
\end{proof}

\begin{lem}
\label{lem:S2}
Let $T_0 \in \mathcal{T}_{n-1}^e,X\subsetneq \Omax(T_0)$ and $T \in \mathfrak{T}(T_0,X)$. The following assertions are equivalent.
\begin{enumerate}
\item The set $\mathfrak{T}(T_0,X)$ is defined in the situation $\mathcal{S}_2$ of Definition \ref{defi:Pimoins1}.
\item The point $p_{n+1}^T$ is free.
\item If $|X| < \max(T_0)+1$, then the last element $p^r$ of $R(T_0)$ belongs to $X$, and no two elements of $\Omax(T_0) \backslash X$ are labeled with $\mathfrak{b}'$ and~$\mathfrak{r}'$.
\end{enumerate}
\end{lem}

\begin{proof}
The equivalence (i) $\Leftrightarrow$ (iii) is straightforward by the third point of Definition \ref{defi:newlabels}.
Suppose that $\mathfrak{T}(T_0,X)$ is defined in the situation $\mathcal{S}_2$ of Definition \ref{defi:Pimoins1}, hence $\mathfrak{T}(T_0,X) = \{T\}$. If $|X| = \max(T_0)+1$, then $\max'(T) = \omax(T_0) - |X| = 1$, so $\Max'(T) = \{p_{n-1}^T\}$ and $p_{n+1}^T$ is free (because it doesn't belong to $\Max'(T)$). If $|X| < \max(T_0)+1$, one of the points of $C_{n-1}^{T_0}$ is labeled with $\mathfrak{r}$ or $\fgg$, so one of the points of $C_{n-1}^T$ is labeled with $\rho$ or $\gamma$. Assume then that $p_{n+1}^T$ is not free, then one of the two points $p_{n-1}^T$ and $p_{n+1}^T$ is an element of $B'(T)$ and the other is an element of $R'(T)$, which by Definition \ref{defi:labels} have labels $\beta$ and $\rho$, and their insertions in $\Pi_0(T)$ plot two points labeled with $\fb'$ and $\fr'$, which contradicts (i). So (i) implies~(ii).

Reciprocally, suppose that $p_{n+1}^T$ is free. Recall that $X = X_T$ has the cardinality $\max(T)$ by Remark \ref{rem:maximalPimoins1}, so if $\max(T) = \max(T_0)+1$ then we indeed have (i). Suppose now that $\max(T) < \max(T_0)+1$. In general, by Proposition \ref{prop:TPathstationnaire}, the roots of the points of $C_{n-1}^T$ belong to $\{n-1,n,n+1,2n\}$. Since $p_{n+1}^T$ is free, these roots cannot be $n+1$. Since $\max'(T) = \omax(T_0)-\max(T) \geq~2$, the point $p_{n-1}^T$ cannot be located in the box $(n-1:n-1)$ (otherwise $\Max'(T)$ would be $\{p_{n-1}^T\}$ because $p_{n+1}^T$ is free, whereas its cardinality exceeds 2). Consequently, if the roots of the points of $C_{n-1}^T$ are $n-1$ and $n$, then the point whose root is $n-1$ is free, which implies that it is an element of $\mathcal{G}(T)$ (because the point whose root is $n$ belongs to $B(T)$), hence it is labeled with $\gamma$ and its insertion in $\Pi_0(T)$ plots a point labeled with $\fgg$ in $C_{n-1}^{T_0}$. Otherwise, the roots of the points of $C_{n-1}^T$ are $n$ and $2n$, in particular the point whose root is $2n$ belongs to $R(T)$ so it is labeled with $\rho$ and its insertion in $\Pi_0(T)$ plots a point labeled with $\fr$ in $C_{n-1}^{T_0}$. In both cases, we have (i).
\end{proof}

The following lemma is a consequence of Lemma \ref{lem:S1} and Lemma \ref{lem:S2}.

\begin{lem}
\label{lem:S3}
Let $T_0 \in \mathcal{T}_{n-1}^e,X\subsetneq \Omax(T_0)$ and $T \in \mathfrak{T}(T_0,X)$. The following assertions are equivalent.
\begin{enumerate}
\item The set $\mathfrak{T}(T_0,X)$ is defined in the situation $\mathcal{S}_3$ of Definition \ref{defi:Pimoins1}.
\item The point $p_{n+1}^T$ is not free and one point of $C_{n-1}^T$ is an element of $R(T)$ or $\mathcal{G}(T)$.
\item The set $X$ contains the last element $p^r$ of $R(T_0)$,$|X| \leq \max(T_0)$, and no two elements of $\Omax(T_0) \backslash X$ are labeled with $\mathfrak{b}'$ and $\mathfrak{r}'$.
\end{enumerate}
\end{lem}

\begin{prop}
\label{prop:cardinalPimoins1}
For all $i \in [0,n-2]$ and $T_0 \in \mathcal{T}_{n-1,i}$, the set
$\Pi^{-1}(T_0) \subset \bigsqcup_{k = 0}^{i+1} \mathcal{T}_{n,k}^e$  is partitioned as follows~:
\begin{itemize}
\item $\Pi^{-1}(T_0) \cap \mathcal{T}_{n,0}^e$ has one unique element, which has $\text{fr}(T_0)$ free points;
\item $\Pi^{-1}(T_0) \cap \mathcal{T}_{n,i+1}^e$ has $i+2$ elements, each of which has $\text{fr}(T_0)+1$ free points;
\item for all $k \in [i]$, let $N_k(T_0)$ be the number of elements $T \in \Pi^{-1}(T_0) \cap \mathcal{T}_{n,k}^e$ such that $p_{n+1}^T$ is free. All these tableaux have $\text{fr}(T_0) +1$ free points. Afterwards, in addition to these tableaux, the set $\Pi^{-1}(T_0) \cap \mathcal{T}_{n,k}^e$ is also composed of~:
\begin{itemize}
\item $\binom{i+1}{k}$ tableaux whose $(n-1)$-th column contains no element of $R(T)$ or $\mathcal{G}(T)$, each of which has $\text{fr}(T_0)$ free points;
\item $2 \left( \binom{i+1}{k-1} - N_k(T_0) \right)$ other tableaux, each of which has $\text{fr}(T_0)$ free points.
\end{itemize}
\end{itemize}
\end{prop}

\begin{proof}
By formula~\eqref{eq:Pimoins1}, the set $\Pi^{-1}(T_0) \cap \mathcal{T}_{n,0}^e$ is $\mathfrak{T}(T_0,\emptyset)$. By Lemma \ref{lem:S1}, this set is defined in the situation $\mathcal{S}_1$ of Definition \ref{defi:Pimoins1}, hence $\mathfrak{T}(T_0,\emptyset)$ has one unique element $T^1$ such that $\text{fr}(T^1) = \text{fr}(T_0)$ in view of Remark \ref{rem:freepointsofantecedents}.

Still by formula~\eqref{eq:Pimoins1}, the set $\Pi^{-1}(T_0) \cap \mathcal{T}_{n,i+1}^e$ is $$\bigsqcup_{X \subset \text{Omax}(T_0),|X|=i+1} \mathfrak{T}(T_0,X),$$
and for every of the $\binom{i+2}{i+1} = i+2$ subsets $X \subset \Omax(T_0)$ whose cardinality is $i+1$, the set $\mathfrak{T}(T_0,X)$ is defined in the situation $\mathcal{S}_2$ of Definition \ref{defi:Pimoins1} in view of Lemma \ref{lem:S2}, so it contains one unique element $T^2$ such that $\text{fr}(T^2) = \text{fr}(T_0)+1$ in view of Remark \ref{rem:freepointsofantecedents}.

Let $k \in [i]$. Consider the tableaux $T \in \Pi^{-1}(T_0) \cap \mathcal{T}_{n,k}^e$ such that $p_{n+1}^T$ is not free and no point of $C_{n-1}^T$ is an element of $R(T)$ or $\mathcal{G}(T)$. By formula~\eqref{eq:Pimoins1} and Lemma \ref{lem:S1}, the set of these tableaux is $\bigsqcup_X \mathfrak{T}(T_0,X)$ where the union is over the subsets $X \subset \Omax(T_0) \backslash \{p^r\}$ (where $p^r$ is the last element of $R(T_0)$) such that $|X| = k$. Since $\omax(T_0) = i+2$, there are $\binom{i+1}{k}$ such subsets, and for every of these subsets $X$, the set $\mathfrak{T}(T_0,X)$ is defined in the situation $\mathcal{S}_1$ of Definition \ref{defi:Pimoins1}, hence $\mathfrak{T}(T_0,\emptyset)$ has one unique element $T^1$ such that $\text{fr}(T^1) = \text{fr}(T_0)$ in view of Remark \ref{rem:freepointsofantecedents}.

Consider now the $N_k(T_0)$ tableaux $T \in \Pi^{-1}(T_0) \cap \mathcal{T}_{n,k}^e$ such that $p_{n+1}^T$ is free. By formula~\eqref{eq:Pimoins1} and Lemma \ref{lem:S2}, the set of these tableaux is $\bigsqcup_X \mathfrak{T}(T_0,X)$ where the union is over the subsets $X \subset \Omax(T_0)$ such that $p^r \in X,|X| = k$ and no two elements of $\Omax(T_0) \backslash X$ are labeled with $\fb'$ and $\fr'$. For every of these subsets $X$, the set $\mathfrak{T}(T_0,X)$ is defined in the situation $\mathcal{S}_2$ of Definition \ref{defi:Pimoins1}, hence $\mathfrak{T}(T_0,X)$ has one unique element $T^2$ such that $\text{fr}(T^2) = \text{fr}(T_0)+1$ in view of Remark \ref{rem:freepointsofantecedents}. Consequently, the number of subsets $X \subset \Omax(T_0)$ that contain $p^r$, whose cardinality is $k$, and such that no two elements of $\Omax(T_0) \backslash X$ are labeled with $\fb'$ and $\fr'$, is exactly $N_k(T_0)$.

Finally, consider the tableaux $T \in \Pi^{-1}(T_0) \cap \mathcal{T}_{n,k}^e$ such that $p_{n+1}^T$ is not free and one of the points of $C_{n-1}^T$ is an element of $R(T)$ or $\mathcal{G}(T)$. In view of the above paragraphs, they are exactly the tableaux $T \in \Pi^{-1}(T_0) \cap \mathcal{T}_{n,k}^e$ defined in the situation $\mathcal{S}_3$ of Definition \ref{defi:Pimoins1}. By Lemma \ref{lem:S3}, the set of these tableaux is $\bigsqcup_X \mathfrak{T}(T_0,X)$ where the union is over the subsets $X \subset \Omax(T_0)$ such that $p^r \in X,|X| = k$, and there exist two elements of $\Omax(T_0) \backslash X$ labeled with $\fb'$ and $\fr'$. Since there exist $\binom{i+1}{k-1}$ subsets $X \subset \Omax(T_0)$ that contain $p^r$ and whose cardinality is $k$, in view of the previous paragraph, the number of these subsets $X$ such that there exist two elements of $\Omax(T_0) \backslash X$ labeled with $\fb'$ and $\fr'$ is exactly $\binom{i+1}{k-1}-N_k(T_0)$. For every of these subsets $X$, the set $\mathfrak{T}(T_0,X)$ is defined in the situation $\mathcal{S}_3$ of Definition \ref{defi:Pimoins1}, hence $\mathfrak{T}(T_0,X)$ has two element $T^{3/}$ and $T^{3 \backslash}$ such that $\text{fr}(T^{3 /}) = \text{fr}(T^{3 \backslash})= \text{fr}(T_0)$ in view of Remark \ref{rem:freepointsofantecedents}, so the total number of tableaux $T \in \Pi^{-1}(T_0) \cap \mathcal{T}_{n,k}^e$ defined in the situation $\mathcal{S}_3$ is exactly $2 \left( \binom{i+1}{k-1} - N_k(T_0) \right)$.
\end{proof}

\begin{cor}
The integers $(e_{n,k})_{0 \leq k \leq n-1}$ satisfy the induction formulas \eqref{eq:an0},\eqref{eq:ank},\eqref{eq:anmoins1}. This proves Theorem \ref{theo:PnintermsofTne} in view of $e_{1,0} = 1 = c_{1,0}$.
\end{cor}

\begin{proof}
Let $n \geq~2$. By Proposition~\ref{prop:Pimoins1}, we have $\mathcal{T}_{n,k}^e \subset \bigsqcup_{i = k-1}^{n-2} \Pi^{-1}(\mathcal{T}_{n-1,i}^e)$ for all $k \in [n-1]$, and~:
\begin{itemize}
\item For all $i \in [0,n-2]$, every tableau $T_0 \in \mathcal{T}_{n-1}^e$ gives birth to $1$ tableau $T \in \mathcal{T}_{n,0}^e$, such that $\text{fr}(T) = \text{fr}(T_0)$, hence $$e_{n,0} = \sum_{T_0 \in \mathcal{T}_{n-1}^e} 2^{\text{fr}(T_0)-1} = \sum_{i = 0}^{n-2} 2^i e_{n-1,i},$$
which is formula \eqref{eq:an0}.
\item For all $k \in [n-2]$,
\begin{enumerate}
\item every tableau $T_0 \in \mathcal{T}_{n-1,k-1}^e$ gives birth to $k+1$ tableaux $T \in \mathcal{T}_{n,k}^e$, such that $\text{fr}(T) = \text{fr}(T_0)+1$, hence
\begin{align*}
\sum_{T \in \mathcal{T}_{n,k}^e,\Pi(T) \in \mathcal{T}_{n-1,k-1}^e} 2^{\text{fr}(T)-1-k} &= (k+1) \sum_{T_0 \in \mathcal{T}_{n-1,k-1}^e} 2^{\text{fr}(T_0)-1-(k-1)}\\
&= (k+1) e_{n-1,k-1};
\end{align*}
\item for all $i \in [k,n-2]$, every tableau $T_0 \in \mathcal{T}_{n-1,i}^e$ gives birth to $\binom{i+1}{k} + 2 \left( \binom{i+1}{k-1} - N_k(T_0) \right)$ tableaux $T \in \mathcal{T}_{n,k}^e$ such that $\text{fr}(T) = \text{fr}(T_0)$, and $N_k(T_0)$ tableaux $T \in \mathcal{T}_{n,k}^e$ such that $\text{fr}(T) = \text{fr}(T_0)+1$, hence the sum $\sum_{T \in \mathcal{T}_{n,k}^e,\Pi(T) \in \mathcal{T}_{n-1,i}^e} 2^{\text{fr}(T)-1-k}$ equals
\begin{align*}
&  \sum_{T_0 \in \mathcal{T}_{n-1,i}^e} \left( \binom{i+1}{k} + 2 \left( \binom{i+1}{k-1} - N_k(T_0) \right) \right) 2^{\text{fr}(T_0)-1-k} + N_k(T_0) 2^{\text{fr}(T_0)-k}\\
=& 2^{i-k} \left( \binom{i+1}{k} + 2 \binom{i+1}{k-1} \right) \sum_{T_0 \in \mathcal{T}_{n-1,i}^e} 2^{\text{fr}(T_0)-1-i}\\
=& 2^{i-k} \left( \binom{i+1}{k} + 2 \binom{i+1}{k-1} \right) e_{n-1,i}.
\end{align*}
\end{enumerate}
The equalities of (i) and (ii) give
$$ \sum_{T \in \mathcal{T}_{n,k}^e} 2^{\text{fr}(T)-1-k} = (k+1) e_{n-1,k-1} + 2^{i-k} \left( \binom{i+1}{k} + 2 \binom{i+1}{k-1} \right) e_{n-1,i}, $$
which is formula \eqref{eq:ank}.
\item Finally note that $\Pi(\mathcal{T}_{n,n-1}^e) \subset \mathcal{T}_{n-1,n-2}^e$, and every tableau $T_0 \in \mathcal{T}_{n-1,n-2}^e$ gives birth to $n$ tableaux $T \in \mathcal{T}_{n,n-1}^e$, such that $\text{fr}(T) = \text{fr}(T_0)+1$, hence
\begin{align*}
e_{n,n-1} = \sum_{T \in \mathcal{T}_{n,n-1}^e} 2^{\text{fr}(T)-1-(n-1)} &= n \sum_{T_0 \in \mathcal{T}_{n-1,k-1}^e} 2^{\text{fr}(T_0)-1-(n-2)}\\
&= n e_{n-1,n-2},
\end{align*}
which is formula \eqref{eq:anmoins1}.
\end{itemize}
This completes the proof.
\end{proof}

\section{Proof of Theorem \ref{theo:PnintermsofTno}}
\label{sec:PnintermsofTno}
In this section, we construct a surjection $\mathcal{P} : \mathcal{T}_n^e \twoheadrightarrow \mathcal{T}_{n-1}^o (n \geq~2)$ such that
\begin{equation}
\label{eq:sumoverPmoins1deT0}
\sum_{T \in \mathcal{P}^{-1}(T_0)} 2^{\text{fr}(T)-1-\max(T)} x^{\max(T)} = 2^{\text{fr}(T_0)-\text{g}(T_0)} x^{\text{v}(T_0)}(1+x)^{\text{g}(T_0)}
\end{equation}
for all $T_0 \in \mathcal{T}_{n-1}^o$, which proves Theorem \ref{theo:PnintermsofTno} in view of Theorem \ref{theo:PnintermsofTne}.

\begin{dfn}
Let $T \in \mathcal{T}_n^e$. We label the elements of $B(T)$ and $R(T)$ with the letter $\nu$, and the elements of $\mathcal{G}(T)$ with the letter $\gamma$.
\end{dfn}

\begin{dfn}[Map $\mathcal{P} : \mathcal{T}_n^e \rightarrow \mathcal{T}_{n-1}^o$]
\label{defi:P}
Let $T \in \mathcal{T}_n^e$. We first define a rectangular tableau $\mathcal{P}_0(T)$, made of $n-1$ columns and $2n-1$ rows, that contains a point in its box $\left( j:\overset{\approx}{i} \right)$ for all points $p = (j:i)$ of $T$ that doesn't belong to $\text{Max}(T)$ (in particular $j < n$), where
$$\overset{\approx}{i} = \begin{cases}
i &\text{if $i <n$},\\
i-1 &\text{otherwise.}
\end{cases}$$
For the example of $T_1 \in \mathcal{T}_7^e$ depicted in Figure~\ref{fig:examplepaths}, we obtain the tableau $\mathcal{P}_0(T_1)$ of Figure~\ref{fig:P0T1}.

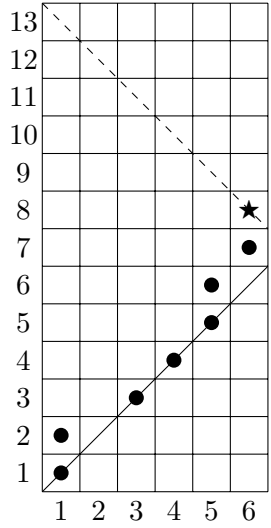
\begin{figure}[!htbp]
\begin{center}

\begin{tikzpicture}[scale=0.5]
\draw (-0.5,0.5) node[scale=1]{$1$};
\draw (-0.5,1.5) node[scale=1]{$2$};
\draw (-0.5,2.5) node[scale=1]{$3$};
\draw (-0.5,3.5) node[scale=1]{$4$};
\draw (-0.5,4.5) node[scale=1]{$5$};
\draw (-0.5,5.5) node[scale=1]{$6$};
\draw (-0.5,6.5) node[scale=1]{$7$};
\draw (-0.5,7.5) node[scale=1]{$8$};
\draw (-0.5,8.5) node[scale=1]{$9$};
\draw (-0.5,9.5) node[scale=1]{$10$};
\draw (-0.5,10.5) node[scale=1]{$11$};
\draw (-0.5,11.5) node[scale=1]{$12$};
\draw (-0.5,12.5) node[scale=1]{$13$};

\draw (0.5,-0.5) node[scale=1]{$1$};
\draw (1.5,-0.5) node[scale=1]{$2$};
\draw (2.5,-0.5) node[scale=1]{$3$};
\draw (3.5,-0.5) node[scale=1]{$4$};
\draw (4.5,-0.5) node[scale=1]{$5$};
\draw (5.5,-0.5) node[scale=1]{$6$};

\draw (0,0) grid[step=1] (6,13);
\draw (0,0) -- (6,6);
\draw [dashed] (0,13) -- (6,7);
\fill (0.5,0.5) circle (0.2);
\fill (0.5,1.5) circle (0.2);
\fill (2.5,2.5) circle (0.2);
\fill (3.5,3.5) circle (0.2);
\fill (4.5,4.5) circle (0.2);
\fill (4.5,5.5) circle (0.2);
\fill (5.5,6.5) circle (0.2);
\draw (5.5,7.5) node[scale=0.8]{$\bigstar$};

\end{tikzpicture}

\end{center}
\caption{The tableau $\mathcal{P}_0(T_1)$.}
\label{fig:P0T1}
\end{figure}

Afterwards, for $j$ from $1$ to $n-1$, if $C_j^T$ contains a point $p$ labeled with $\nu$ (respectively $\gamma$), we insert a point in the box $(j:2n-1)$ of $\mathcal{P}_0(T)$ (respectively the box $(j:j)$), following Definition~\ref{defi:insertion}.

Let $\mathcal{P}(T)$ be the tableau obtained at the end of this algorithm. It is straightforward by Definition \ref{defi:insertion} that $\mathcal{P}(T)$ is an element of $\mathcal{T}_{n-1}^o$.
\end{dfn}

For the example of $T = T_1 \in \mathcal{T}_7^e$ (see Figure~\ref{fig:examplepaths}), the tableau $\mathcal{P}(T_1) \in \mathcal{T}_6^o$ is as depicted in Figure~\ref{fig:examplepaths2}.

Let now $T_0 \in \mathcal{T}_{n-1}^o$. To construct $\mathcal{P}^{-1}(T_0)$, we consider four words $b,r,br,bg$ with which we will label the columns of $T_0$ that contain elements of $V(T_0)$ (recall that if a point of a column of $T_0$ belongs to $\mathcal{G}(T_0)$, then the other point of this column belongs to $V(T_0)$). Let
\begin{align*}
J(T_0) &= \{j \in [n-1] : \text{$C_j^{T_0}$ contains an element of $V(T_0)$}\},\\
J_g(T_0) &= \{j \in [n-1] : \text{$C_j^{T_0}$ contains an element of $\mathcal{G}(T_0)$}\}.
\end{align*}
We have $J_g(T_0) \subset J(T_0), |J(T_0)| = v(T_0),|J_g(T_0)| = g(T_0)$. Consider the set $L(T_0)$ of the functions $l : J(T_0) \rightarrow \{b,r,bg,br\}$ such that $l(j) \in \{b,r\}$ if $j \not\in J_g(T_0)$, otherwise $l(j) \in \{br,bg,r\}$. The cardinality of $L(T_0)$ is $|L(T_0)| = 2^{v(T_0)-g(T_0)} 3^{g(T_0)}$.

\begin{dfn}
\label{defi:Pmoins1}
For all $l \in L(T_0)$, we define a tableau $U^l(T_0) \in \mathcal{T}_n^e$ as follows. First consider a rectangular tableau $U_0^l(T_0)$ made of $n$ columns and $2n$ rows, that contain points in its boxes $\left( j : \overline{\overline{i}} \right)$ for all points $(j:i)$ of $T_0$ that doesn't belong to $V(T_0)$ or $\mathcal{G}(T_0)$, where $\overline{\overline{i}}$ is defined as $i$ if $i < n$, as $i+1$ otherwise. Let us set $T = U_0^l(T_0)$. For $j$ from $1$ to $n-1$, assume that $C_j^{T_0}$ contains an element of $V(T_0)$ (otherwise $C_j^T$ already contains two points and we do nothing).
\begin{itemize}
\item If no point of $C_j^{T_0}$ belongs to $\mathcal{G}(T_0)$, we have $j \in J(T_0) \backslash J_g(T_0)$, and $l(j) \in \{b,r\}$. If $l(j) = b$ (respectively $l(j) = r$), then we insert a point in the box $(j:n)$ (respectively $(j:2n)$) of $T$.
\item Otherwise $j \in J_g(T_0)$ and $l(j) \in \{br,bg,r\}$. If $l(j) = br$ (respectively $bg,r$), then we insert two points in the boxes $(j:n)$ and $(j:2n)$ (respectively $(j:n)$ and $(j:j)$, $(j:2n)$ and $(j:j)$) of $T$.
\end{itemize}
At the end of the algorithm, we insert two points in the boxes $(n:n)$ and $(n:2n)$ of $T$. By Definition \ref{defi:insertion}, the tableau $T$ has become an element of $\mathcal{T}_n^e$, that we denote by $U^l(T_0)$.
\end{dfn}

For example, the tableau $T_1 \in \mathcal{T}_7^e$ depicted in Figure~\ref{fig:examplepaths} is $U^l(\mathcal{P}(T_1))$, where recall that $\mathcal{P}(T_1) \in \mathcal{T}_6^o$ is the tableau depicted in Figure~\ref{fig:examplepaths2}, and where $l = \begin{pmatrix}
2 & 3 & 4 \\ bg & b & r
\end{pmatrix}$.

\begin{prop}
\label{prop:Pmoins1}
The set $\mathcal{P}^{-1}(T_0)$ is exactly $\{U^l(T_0) : l \in L(T_0)\}$ (so $\mathcal{P}$ is surjective).
\end{prop}

\begin{proof}
Let $T \in \mathcal{P}^{-1}(T_0)$. By Definition \ref{defi:P}, for all $j \in [n-1]$, we have $j \in J(T_0)$ (respectively $j \in J_g(T_0)$) if and only if one of the points of $C_j^T$ belongs to $\Max(T)$ (respectively to $\mathcal{G}(T)$). We then define $l_T \in L(T_0)$ by
$$l_T(j) = \begin{cases}
b &\text{if $C_j^T$ has a point of $B(T)$ and no point of $R(T)$ or $\mathcal{G}(T)$},\\
br &\text{if $C_j^T$ has a point of $B(T)$ and a point of $R(T)$},\\
bg &\text{if $C_j^T$ has a point of $B(T)$ and a point of $\mathcal{G}(T)$},\\
r &\text{if $C_j^T$ has a point of $R(T)$ and no point of $B(T)$}
\end{cases}$$
for all $j \in J(T_0)$. It is easy to check that $T = U^{l_T}(T_0)$, and that $\mathcal{P}(U^l(T_0)) = T_0$ for all $l \in L(T_0)$.
\end{proof}

\begin{prop}
\label{prop:sumoverPmoins1deT0}
Formula~\eqref{eq:sumoverPmoins1deT0} is true.
\end{prop}

\begin{proof}
Let $l \in L(T_0)$ and $T = U^l(T_0) \in \mathcal{P}^{-1}(T_0)$. Let us determine $\text{fr}(T)$ and $\max(T)$ in terms of $T_0$ and $l$. First of all, note that in general all the elements of $V(T_0)$ and $\mathcal{G}(T_0)$ are free points of $T_0$. Let $\text{fr}_0(T_0) = \text{fr}(T_0) - v(T_0) - g(T_0)$, by Definition \ref{defi:Pmoins1} it is the number of free points of $T$ that were initially in $U_0^l(T_0)$. Let $\delta_l \in [0,g(T_0)]$ be the number of elements $j \in J_g(T_0)$ such that $l(j) = r$. By construction of $T$, for all of these $\delta_l$ integers $j$, the two points of $C_j^{T_0}$ are an element of $V(T_0)$ and an element of $\mathcal{G}(T_0)$ whose insertions in $U_0^l(T_0)$ plot two free points, the first of which belongs to $R(T)$, the second of which doesn't belong to $\Max(T)$. Consequently, by Remark \ref{rem:evenmaximalfreepoints}, we have
\begin{align*}
b(T)+r(T) &= v(T_0),\\
g(T) & = g(T_0)-\delta_l,\\
\max(T) & = b(T)+r(T)+g(T)\\& = v(T_0)+g(T_0)-\delta_l,\\
\text{fr}(T) &= \text{fr}_0(T_0) + (\max(T)+1)+\delta_l  \\ & = \text{fr}(T_0)+1,
\end{align*}
so that
\begin{equation}
\label{eq:TT0}
2^{\text{fr}(T)-1-\max(T)} x^{\max(T)} = 2^{\text{fr}(T_0)-v(T_0)-g(T_0)+ \delta_l} x^{v(T_0)+g(T_0)- \delta_l}.
\end{equation}
Now, for all $\delta \in [0,g(T_0)]$, the number of elements $l \in L(T_0)$ such that $\delta_l = \delta$ is $\binom{g(T_0)}{\delta} 2^{g(T_0)-\delta} 2^{v(T_0)-g(T_0)} = \binom{g(T_0)}{\delta} 2^{v(T_0)-\delta}$, so, by Proposition \ref{prop:Pmoins1} and formula~\eqref{eq:TT0}, the sum $\sum_{T \in \mathcal{P}^{-1}(T_0)} 2^{\text{fr}(T)-1-\max(T)} x^{\max(T)}$ equals
\begin{align*}
&\sum_{\delta =0}^{g(T_0)} \binom{g(T_0)}{\delta} 2^{v(T_0)-\delta} 2^{\text{fr}(T_0)-v(T_0)-g(T_0)+ \delta} x^{v(T_0)+g(T_0)- \delta}\\& = 2^{\text{fr}(T_0)-g(T_0)} x^{v(T_0)+g(T_0)} \sum _{\delta = 0}^{g(T_0)} \binom{g(T_0)}{\delta} x^{-\delta}\\
& = 2^{\text{fr}(T_0)-g(T_0)} x^{v(T_0)+g(T_0)} \left( 1 + \frac{1}{x} \right)^{g(T_0)},
\end{align*}
hence formula~\eqref{eq:sumoverPmoins1deT0}.
\end{proof}

This proves Theorem \ref{theo:PnintermsofTno}.

\section*{Acknowledgments}
The work of EF was partially supported by the grant RSF-DFG 16-41-01013.


\begin{thebibliography}{99}

\bibitem[BD]{BD}
D.~Barsky, D.~Dumont, {\it Congruences pour les nombres de Genocchi de 2e esp\`ece},
Groupe d'\'etude d'Analyse ultram\'etrique,
8e ann\'ee, no. 34, 1980/81, 13 pp.

\bibitem[B1]{B1}
A.~Bigeni, {\em Enumerating the symplectic Dellac configurations}, arXiv:1705.03804.

\bibitem[B2]{B2}
A.~Bigeni, {\em A generalization of the Kreweras triangle through the universal $sl_2$ weight system},
arXiv:1712.05475.

\bibitem[B3]{B3}
A.~Bigeni, {\em Combinatorial interpretations of the Kreweras triangle in terms of subset tuples},
arXiv:1712.01929.

\bibitem[B4]{B4}
A.~Bigeni, {\em Combinatorial study of the Dellac configurations and the q-extended normalized median Genocchi numbers},
Electronic Journal of Combinatorics, 2014, Vol. 21, No. 2. 
 
\bibitem[BF]{BF}
A.~Bigeni, E.~Feigin, {\em Symmetric Dellac configurations and symplectic/orthogonal flag varieties}, arXiv:1804.10804.    

\bibitem[De]{De} H.~Dellac, Problem 1735, L'Interm\'ediaire des Math\'ematiciens, 7 (1900), 9--10.

\bibitem[Du]{Du} D.~Dumont, {\em Interpr\'etations combinatoires des nombres de Genocchi}, Duke
Math. J. 41 (1974), 305--318.

\bibitem[DR]{DR}
D.~Dumont, A.~Randrianarivony,
{\it D\'erangements et nombres de Genocchi}, Discrete Math. 132 (1994), 37--49.

\bibitem[DZ]{DZ}
D.~Dumont, J.~Zeng, {\it Further results on Euler and Genocchi numbers},
Aequationes Mathemicae 47 (1994), 31--42.

\bibitem[FF]{FF}  X.~Fang, G.~Fourier, 
{\em Torus fixed points in Schubert varieties and normalized
median Genocchi numbers}, S\'em. Lothar. Combin. 75 (2015), Art. B75f, 12 pp.


\bibitem[F1]{F1} E.~Feigin, {\em Degenerate flag varieties and the median Genocchi numbers}, Math.
Res. Lett. 18 (2011), no. 6, 1163--1178.

\bibitem[F2]{F2} E.~Feigin, {\em The median Genocchi numbers, q-analogues and continued fractions},
European J. Combin., 33(8), 1913--1918, 2012.

\bibitem[F3]{F3}
E.~Feigin, {\em ${\mathbb G}_a^M$ degeneration of flag varieties},  Selecta Mathematica, 
New Series, vol. 18 (2012), no. 3, pp. 513--537.


\bibitem[FFL]{FFL} E.~Feigin, M.~Finkelberg, P.~Littelmann, {\em Symplectic degenerate flag varieties},
Canad. J. Math. 66 (2014), no. 6, 1250--1286.


\bibitem[HZ1]{HZ1}
G.-N. Han and J. Zeng, {\it On a $q$-sequence that generalizes the median Genocchi numbers},
Ann. Sci. Math. Qu\'ebec 23 (1999), 63--72.

\bibitem[HZ2]{HZ2}
G.~Han, J.~Zeng, {\it q-Polynomes de Gandhi et statistique de Denert}, Discrete Math.
205 (1999), no. 1--3, 119-143.

\bibitem[K]{K}
G. Kreweras, {\it Sur les permutations compt\'ees par les nombres de Genocchi de 1-i\`ere et
2-i\`eme esp\`ece}, Europ. J. Combinatorics 18 (1997), 49--58.

\bibitem[OEIS1]{112738} The On-Line Encyclopedia of Integer Sequences, https://oeis.org/A000366.

\bibitem[OEIS2]{11321267} The On-Line Encyclopedia of Integer Sequences, https://oeis.org/A098278.

\bibitem[OEIS3]{121098} The On-Line Encyclopedia of Integer Sequences, https://oeis.org/A098279.

\bibitem[OEIS4]{114461024} The On-Line Encyclopedia of Integer Sequences, http://oeis.org/A000657.

\bibitem[OEIS5]{1324402} The On-Line Encyclopedia of Integer Sequences, http://oeis.org/A002832.


\bibitem[RZ]{RZ} A.~Randrianarivony, J.~Zeng, {\em Une famille de polyn\^omes qui interpole plusieurs
suites classiques de nombres}, Adv. in Appl. Math. 17 (1996), no. 1, 1--26.   

\bibitem[ZZ]{ZZ} 
J.~Zeng, J.~Zhou, {\it A q-analog of the Seidel generation of Genocchi numbers}. Eur. J. Comb., 2006, pp. 364~381 



\end{thebibliography}
\end{document}